\documentclass[a4paper,reqno, 12pt]{amsart}
\usepackage[left=2.8cm,right=2.8cm,top=2.86cm,bottom=2.86cm]{geometry}
\usepackage{amsmath,amssymb,latexsym,esint,cite,verbatim,wasysym,mathrsfs}
\usepackage{microtype}
\usepackage{color,graphicx}
\usepackage[colorlinks=true,urlcolor=blue, citecolor=red,linkcolor=blue,
linktocpage,pdfpagelabels,bookmarksnumbered,bookmarksopen]{hyperref}
\usepackage[hyperpageref]{backref}
\usepackage[english]{babel}
\usepackage{tikz}
\usepackage[shortlabels]{enumitem}

\renewcommand{\epsilon}{\varepsilon}

\newcommand{\pnorm}[2][]{\if #1'' \left|#2\right|_p \else \left|#2\right|_{#1} \fi}
\hypersetup{linkcolor=blue, colorlinks=true ,citecolor = red}

\newcommand*\diff{\mathop{}\!\mathrm{d}}

\newcommand{\norm}[1]{|\!|#1|\!|}

\newcommand{\R}{{\mathbb R}}

\newcommand{\N}{{\mathbb N}}

\newcommand{\Assg}[1]{\textup{(g)}}

\begin{document}
	\title [Existence results to generalized $p(\cdot)$-Laplace equations]
	{On sufficient ``local'' conditions for existence results to generalized $p(\cdot)$-Laplace equations involving critical growth}
	
	\author[K. Ho]{Ky Ho}
	\address{Ky Ho\newline
		Institute of Applied Mathematics, University of Economics Ho Chi Minh City, 59C, Nguyen Dinh Chieu St., Dist. 3, Ho Chi Minh City, Vietnam}
	\email{kyhn@ueh.edu.vn}
	
	\author[I. Sim]{Inbo Sim}
	\address{Inbo Sim \newline
		Department of Mathematics, University of Ulsan, Ulsan 44610, Republic of Korea}
	\email{ibsim@ulsan.ac.kr}

	
	\begin{abstract}
		In this paper, we study the existence of multiple solutions to a generalized $p(\cdot)$-Laplace equation with two parameters involving critical growth. More precisely, we give sufficient ``local'' conditions, which mean that growths between the main operator and nonlinear term are locally assumed for the cases $p(\cdot)$-sublinear, $p(\cdot)$-superlinear, and sandwich-type. Compared to constant exponent problems (for examples, $p$-Laplacian and $(p,q)$-Laplacian), this characterizes the study of variable exponent problems. We show this by applying variants of the Mountain Pass Theorem for $p(\cdot)$-sublinear and $p(\cdot)$-superlinear  cases and constructing critical values defined by a minimax argument in the genus theory for sandwich-type case. Moreover, we also obtain a nontrivial nonnegative solution for sandwich-type case changing a role of parameters.  Our work is a generalization of several existing works in the literature.
	\end{abstract}
	
	\maketitle
	\numberwithin{equation}{section}
	\newtheorem{theorem}{Theorem}[section]
	\newtheorem{lemma}[theorem]{Lemma}
	\newtheorem{proposition}[theorem]{Proposition}
	\newtheorem{corollary}[theorem]{Corollary}
	\newtheorem{definition}[theorem]{Definition}
	\newtheorem{example}[theorem]{Example}
	\newtheorem{remark}[theorem]{Remark}
	\allowdisplaybreaks
	
	\section{Introduction and Main Results}\label{Intro}
	In this paper we are concerned with the existence and multiplicity of solutions to the following problem:
	\begin{align}\label{Eq.General}
		\begin{cases}
			-\operatorname{div}a(x,\nabla u)=\lambda f(x,u)+\theta b(x)|u|^{t(x)-2}u \quad \text{in } \Omega,\\
			u=0\quad \text{on } \partial \Omega, 
		\end{cases}
	\end{align}
where $\Omega $ is a bounded domain in $\mathbb{R}^{N}$ with a Lipschitz boundary $\partial \Omega;$ $a : \Omega \times \mathbb R^N \to
\mathbb R^N$ is a Carath\'{e}odory function which generalizes  $|\xi| ^{p(x)-2}\xi$, where $p\in  C_+(\overline{\Omega})$ with
$$C_+(\overline{\Omega}) :=\{ r\in C(\overline{\Omega }): 1<r^-:=\min_{x\in \overline{\Omega }} r(x)\leq r^+:=\max_{x\in \overline{\Omega }} r(x)<\infty\};$$
 $f: \Omega \times \mathbb R \to
\mathbb R$ is a Carath\'{e}odory function having subcritical growth; $t\in C_+(\overline{\Omega})$ with $p^+<t^-\leq t(x)\leq p^\ast(x)$ for all $x\in\overline{\Omega}$ such that 
$$\mathcal{C}:=\{x\in\overline{\Omega}:\, t(x)=p^\ast(x)\}\ne\emptyset,$$
 where $p^\ast(x):=\frac{Np(x)}{N-p(x)}$ if $p^+<N$ and $p^\ast(x):=p_0(x)$ for any fixed $p_0\in C_+(\overline{\Omega})$ with $p(x)<p_0(x)$ for all $x\in\overline{\Omega}$ if $p^+\geq N$; $b\in L^\infty(\Omega)$ and $b(x)>0$ for a.e. $x\in\Omega;$ and $\lambda,\theta$ are parameters. 
 
 Throughout this paper, we always assume the following conditions on $a$:
	\begin{itemize}
		\item[$\textup{(A0)}$]
		$a(x,-\xi)=-a(x,\xi)$ for a.e. $x\in \Omega $ and all $\xi \in \mathbb R^N.$
		
		\item[$(\textup{A1})$] There exists a Carath\'{e}odory function $A : \Omega \times \mathbb R^N \to \mathbb R,$ continuously differentiable with respect to its second variable, such that $A(x,0)=0$ for a.e. $x\in \Omega$ and $a(x,\xi)=\nabla_\xi A(x,\xi)$ for a.e. $x\in \Omega $ and all $\xi \in \mathbb R^N.$
		
		\item[($\textup{A2}$)] $|a(x,\xi)| \leq C \left[1 + |\xi|^{p(x)-1}\right]$
		for a.e. $x\in \Omega $ and all $\xi \in \mathbb R^N ,$ where $C$ is a positive constant. 
		\item[($\textup{A3}$)] $|\xi|^{p(x)} \leq a(x,\xi)\cdot \xi\leq p(x) A(x,\xi)$  for a.e. $x\in \Omega $ and all $\xi \in \mathbb R^N$. 
		
		\item[($\textup{A4}$)] $0 < [a(x,\xi) - a(x,\eta)]\cdot (\xi-\eta)$ for a.e. $x\in \Omega $ and all\  $\xi, \eta \in \mathbb R^N$, $\xi \ne \eta.$
	\end{itemize}
The operator satisfying $(\textup{A0})-(\textup{A4})$ is of Leray-Lions type and typical examples are  $$\operatorname{div}\left(|\nabla u|^{p(x)-2}\nabla u\right)\  \text{and}\ \operatorname{div}\left[(1+|\nabla u|^2)^{(p(x)-2)/2}\nabla u\right],$$
which we call the $p(\cdot)$-Laplacian and the generalized mean curvature operator, respectively. Another remarkable example is 
$$\operatorname{div}\left[\sum_{i=1}^n |\nabla u|^{p_i(x)-2}\nabla u\right],$$
which satisfies $(\textup{A0})-(\textup{A4})$ with  $p(x) := \max_{1\le i \le n}  p_i(x),$ and we call this operator the multiple $p(\cdot)$-Laplacian. A special case of the multiple $p(\cdot)$-Laplacian is the operator $\operatorname{div}\left(|\nabla u|^{p-2}\nabla u+|\nabla u|^{q-2}\nabla u\right)$ with $1<q\leq p<\infty$, which is called the $(p,q)$-Laplacian and has received a great attention from mathematicians in the last decade.

\par
The study of $p(\cdot)$-Laplacian was initiated to describe some materials (for example, electrorheological fluids) that are inadequate to deal with in a constant $p(\cdot)$.  Indeed
R\v adulescu \cite{Radul} gave a couple of examples that required the setting of $p(\cdot)$-Laplacian problems that portray phenomenon strongly depending on the position. Thus, it is natural to ask for sufficient effects of the ration of $p(\cdot)$ and growth of a nonlinear term in a local sense.  It is worth mentioning that the set $\mathcal{C}$ is the other condition of locality. Meanwhile, a local effect was also assumed in this study when $p(\cdot)$ is even constant in the sequel papers  \cite{de Figueiredo.2003, de Figueiredo.2006, de Figueiredo.2009}. More precisely, the authors were concerned with the ``local'' growth of nonlinearities,  investigating the existence, nonexistence, and multiplicity of solutions for the family of problems
		\begin{align*}
		\begin{cases}
			-\operatorname{div}\left(|\nabla u|^{p-2}\nabla u\right)=f_\lambda (x,u) \quad \text{in } \Omega,\\
			u>0 \quad \text{in } \Omega,\\
			u=0\quad \text{on } \partial \Omega, 
		\end{cases}
	\end{align*}
where $f$ is locally ``$p$-sublinear'' at 0 and is locally ``$p$-superlinear'' at $\infty$, namely (roughly speaking): 
$$\lim_{s\to 0^+}\frac{f_\lambda (x,s)}{s^{p-1}}=\infty$$
for $x$ in a subdomain $\Omega_1$ of $\Omega$ and 
$$\lim_{s\to \infty}\frac{f_\lambda (x,s)}{s^{p-1}}=\infty$$
for $x$ in a subdomain $\Omega_2$ of $\Omega$. We emphasize that such local assumptions are effective for a nonautonomous nonlinear term and does not allow higher or lower growth than $p$, respectively.
Recently,  Komiya-Kajikiya \cite{Kom-Kaj} studied the existence of infinitely many solutions to problem \eqref{Eq.General} when $a(x,\xi)=|\xi|^{p-2}\xi+|\xi|^{q-2}\xi$ with $1<q\leq p<\infty$ and $\theta=0$. Indeed, they obtained a sequence of solutions with $C^1$-norms converging to 0 (resp. $\infty$) when $f$ satisfies a locally $q$-sublinear condition, i.e., $\inf_{x\in B_\epsilon(x_0)} \frac{F(x,\tau)}{|\tau|^q}\to\infty$  as $\tau\to 0$ (resp. a locally $p$-superlinear condition, i.e., $\inf_{x\in B_\epsilon(x_0)} \frac{F(x,\tau)}{|\tau|^p}\to\infty$  as $|\tau|\to \infty$) for some ball $B_\epsilon(x_0)\subset\Omega$. Here and in what follows, $F(x,\tau):=\int_{0}^{\tau}f(x,s)\diff s$ and $B_\epsilon(x_0)$ denotes a ball in $\mathbb{R}^N$ centered at $x_0$ with radius $\epsilon.$ It is worth mentioning that they did not treat the case of growth between $p$ and $q$, which we call sandwich-type growth for the $(p,q)$-problem. These results have recently been obtained for problem \eqref{Eq.General} under $(\textup{A0})-(\textup{A4})$ in \cite{HHS}. Note that both \cite{Kom-Kaj} and \cite{HHS} only considered subcritical problems. The existence of infinitely many  solutions to critical problems involving the $p$-Laplacian for this kind of locally $p$-superlinear condition was investigated in \cite{Sil-Xav.2003}.

\par
Regarding critical and sandwich-type growth, in \cite{HS.AML21} we considered problem \eqref{Eq.General} of the form:
\begin{align}\label{Sand}
	\begin{cases}
		-\operatorname{div}\left(|\nabla u|^{p-2}\nabla u+|\nabla u|^{q-2}\nabla u\right)=\lambda m(x)|u|^{s-2}u+\theta b(x)|u|^{p^\ast-2}u \quad \text{in } \Omega,\\
		u=0\quad \text{on } \partial \Omega, 
	\end{cases}
\end{align}
where $1<q<s<p$ and the weight $m$ is possibly sign-changing. We obtained the existence of nontrivial nonnegative solutions to \cite{BBF.2021} for a fixed $\lambda$ and some range of parameter $\theta$ (depending on $\lambda$). In \cite{BBF.2021}, Baldelli et al. obtained the multiplicity of solutions for problem \eqref{Sand} when $\Omega=\R^N$ in a reverse way, that is, for $\theta$ fixed and small, they find some range of $\lambda$. The subcritical problem of sandwich-type growth involving variable exponent was studied by Mihailescu-R\v adulescu in \cite{Mih-Rad}. In that paper, the authors studied \eqref{Eq.General} of the form: 
\begin{align}\label{Eq.MR}
\begin{cases}
	-\operatorname{div}\left(|\nabla u| ^{p(x)-2}\nabla u+|\nabla u| ^{q(x)-2}\nabla u\right)=\lambda |u|^{s(x)-2}u \quad \text{in } \Omega,\\
	u=0\quad \text{on } \partial \Omega, 
\end{cases}
\end{align}
with $p,q,s\in C_+(\overline{\Omega})$ satisfying $q^+<s^-\leq s^+<p^-$ and $s^+<q^\ast(x)$ for all $x\in \overline{\Omega}$ and showed that there exists $\lambda_*>0$ such that problem~\eqref{Eq.MR} admits a nontrivial solution for any $\lambda\in [\lambda_*,\infty)$.

Motivated by \cite{BBF.2021,Figueredo2013, de Figueiredo.2003, de Figueiredo.2006, de Figueiredo.2009, Kom-Kaj, HS.AML21, Mih-Rad,Sil-Xav.2003}, our aim is to show the existence of 
  multiple solutions to problem \eqref{Eq.General}, which may include a critical term and generalize previous results, in particular \cite{HS.AML21, Mih-Rad, Sil-Xav.2003} when $f$ satisfies one of three types of growth conditions: locally $p(\cdot)$-sublinear, locally $p(\cdot)$-superlinear, and sandwich-type. Moreover,  we also obtain a nontrivial nonnegative solution for the sandwich-type case by changing the role of the parameters.
  We look for solutions to problem~\eqref{Eq.General} in $W_0^{1,p(\cdot)}(\Omega)$, that is the completion of $C_c^\infty(\Omega)$ with respect to the norm
  $$\|u\|:=\|\nabla u\|_{L^{p(\cdot)}(\Omega)}.$$
 Note that if $p^+<N$ and $p$ is logarithmic H\"older continuous (notation: $p\in C^{0, \frac{1}{|\log t|}}(\overline{\Omega})$), namely, 
 \begin{align*}
 	|p(x)-p(y)| \leq \frac{C}{|\log |x-y||}\quad\text{for all } x,y\in \overline{\Omega} \text{ with } 0<|x-y|<\frac{1}{2},
 \end{align*}
 then the following critical imbedding holds
 \begin{equation}\label{cri.imb}
 	W_0^{1,p(\cdot)}(\Omega)\hookrightarrow L^{p^*(\cdot)}(\Omega)
 \end{equation}
 (see Section~\ref{Pre} for the definitions and properties of the variable exponent Lebesgue-Sobolev spaces). Since we mainly focus on critical problems, throughout this paper we always assume $p\in C^{0, \frac{1}{|\log t|}}(\overline{\Omega})$	when $p^+<N$ in our main results.

Before giving a statement of the main results, we note  that under $(\textup{A1})$, we have that for a.e. $x\in \Omega$ and for all $\xi \in \mathbb R^N$,
	\begin{equation}\label{formA}
		A(x,\xi)=\int_{0}^{1}a(x,t\xi)\cdot \xi \diff t.
	\end{equation}
Thus, $(\textup{A2})$ yields
\begin{equation}\label{A2'}
	|A(x,\xi)| \leq \widetilde{C}\left[1+|\xi|^{p(x)}\right]\ \  \text{for a.e.}\ x\in \Omega\ \text{and all}\ \xi \in \mathbb R^N,
\end{equation}
where $\widetilde{C}$ is a positive constant. Also, from  $(\textup{A2})$ we get
\begin{equation}\label{A2''}
	|a(x,\xi)|^{\frac{p(x)}{p(x)-1}}\leq \bar{C} \left[1+|\xi|^{p(x)}\right]
\end{equation}
for a.e. $x\in \Omega$ and all $\xi \in \mathbb R^N$. 
  In what follows, for $r\in C(\overline{\Omega})$ with $r(x)\geq 1$ for all $x\in\Omega$ and an open set $D\subseteq\Omega$, denote
$$r_D^-:=\inf_{x\in D}r(x),\ r_D^+:=\sup_{x\in D}r(x)$$
and
$$L_+^{r(\cdot)}(D):=\{u\in L^{r(\cdot)}(D):\, u>0\ \text{a.e. in}\ D\}.$$ 

First, let us start with the locally $p(\cdot)$-sublinear case. In this case, we additionally assume for the operator $a$ that:
\begin{itemize}
	\item[($\textup{A5}$)] For each ball $B\subset \Omega$, there exist constants $C_B,\delta_B>0$, and $q_B\in [1,\infty)$ such that
	$$|a(x,\xi)| \leq C_B |\xi|^{q_B-1} \ \text{
		for a.e.} \ x\in B \ \text{and all} \ |\xi|<\delta_B.$$ 
\end{itemize}
It is clear that the multiple $p(\cdot)$-Laplacian and the generalized mean curvature operator above also satisfy ($\textup{A5}$). Note that $(\textup{A3})$ and $(\textup{A5})$ yield $q_B\leq p_B^-$ and moreover, from $(\textup{A5})$ and \eqref{formA} we obtain
\begin{equation}\label{A5'}
	|A(x,\xi)| \leq \frac{C_B}{q_B} |\xi|^{q_B}\  \text{for a.e.} \ x\in B \ \text{and all}\ |\xi|<\delta_B.
\end{equation} 
Assume that $f$ satisfies the following conditions.
 \begin{itemize}
 	\item [$\textup{(F1)}$] $f:\, \Omega\times [-\epsilon_0,\epsilon_0]\to\mathbb{R}$ is a Carath\'{e}odory function with an $\epsilon_0>0$ such that $f$ is odd with respect to the second variable and $d(\cdot) := \sup_{|\tau|\leq\epsilon_0}\, |f(\cdot,\tau)|\in L_+^{\sigma(\cdot)}(\Omega)$ with $\sigma\in C_+(\overline{\Omega})$ satisfying $\sigma(x)>\max\left\{1,\frac{N}{p(x)}\right\}$ for all $x\in \overline{\Omega}.$
 	\item [$\textup{(F2)}$] There exist a ball $B\subset\Omega$ and $m\in L_+^{1}(B)$ such that 
 	$$\lim_{\tau\to 0}\frac{F(x,\tau)}{m(x)|\tau|^{q_{B}}}=\infty\ \ \text{uniformly for a.e.}\ x\in B,$$
 	where $q_{B}$ is determined as in $(\textup{A5})$.
 	 \end{itemize}
  The existence result for the locally $p(\cdot)$-sublinear case is stated as follows. 
\begin{theorem}\label{Theo.Sublinear}
	Let $(\textup{A0})-(\textup{A5}),  \textup{(F1)}$ and $\textup{(F2)}$ hold. Then, for any $\lambda>0$ and $\theta\geq 0$, problem \eqref{Eq.General} admits a sequence of nontrivial solutions $\{u_n\}_{n=1}^\infty$ converging to zero in $W^{1,p(\cdot)}_0(\Omega)\cap L^\infty(\Omega)$. 
\end{theorem}

\par
Next, we study problem \eqref{Eq.General} with a nonlinearity $f$ of locally $p(\cdot)$-superlinear type growth. In this case, we need the following additional assumption on the exponent $p$:
 \begin{itemize}
	\item [$\textup{(P)}$] There exists $w\in L_+^{\frac{p^*(\cdot)}{p^*(\cdot)-p(\cdot)}}(\Omega)$ such that
	\begin{equation}\label{Cond}
	\mu_1:=\inf_{\varphi\in C_c^\infty(\Omega)} \frac{\int_\Omega|\nabla \varphi|^{p(x)}\diff x}{\int_\Omega w(x)|\varphi|^{p(x)}\diff x} \in (0,\infty).
	\end{equation}
\end{itemize}
Obviously, condition $\textup{(P)}$ automatically holds true when $p(\cdot)$ is constant due to the H\"older and Sobolev inequalities. For variable exponent case, $\textup{(P)}$ is fulfilled if there is a vector $x_0\in \mathbb{R}^N\setminus\{0\}$ such that for any $x\in \Omega$, $h(\tau):=p(x + \tau x_0)$ is monotone for $\tau\in I_x= \{\tau:\, x + \tau x_0 \in \Omega\}$ (by taking $w\in L_+^\infty(\Omega)$), see  \cite[Theorems 3.2-3.3]{Fan-Zhang-Zhao.2005}. It is worth pointing out that $\mu_1$ given by \eqref{Cond} is zero if there is an open ball $B_{\epsilon_0}(x_0)\subset\Omega$ such that $p(x_0)<$ (or $>$) $p(x)$ for all $x\in\partial B_{\epsilon_0}(x_0)$, see \cite[Proof of Theorem 3.1]{Fan-Zhang-Zhao.2005}.

Now, we state conditions on $f$ for the locally $p(\cdot)$-superlinear case.
	\begin{itemize}
		\item [$\textup{(F3)}$] $f:\, \Omega\times\mathbb{R}\to\mathbb{R}$ is a Carath\'{e}odory function such that $f$ is odd with respect to the second variable.
	
		\item [$\textup{(F4)}$] There exist functions $r_j, a_j$ with $r_j \in C_+(\overline{\Omega})$, $\underset{x\in\overline{\Omega}}{\inf}[t(x)-r_j(x)]>0,$   $a_j\in L_+^{\frac{t(\cdot)}{t(\cdot)-r_j(\cdot)}}\left(b^{-\frac{r_j}{t-r_j}},\Omega\right)$ ($j=1,\cdots,m_0$), and $\underset{1\leq j\leq m_0}{\max}\, r_j^+>p^-$ such that
		$$|f(x,\tau)|\leq \sum_{j=1}^{m_0} a_j(x)|\tau|^{r_j(x)-1}\ \  \text{for a.e.}\ x\in\Omega \ \text{and all}\ \tau\in\mathbb{R}.$$
			\item [$\textup{(F5)}$] There exist an open ball $B\subset\Omega$ and a function $m\in L_+^{1}(B)$ such that  
			$$\lim_{|\tau|\to \infty}\frac{F(x,\tau)}{m(x)|\tau|^{p_B^+}}=\infty\  \text{uniformly for a.e.} \ x\in B.$$

		\item [$\textup{(F6)}$] For $w$ given by $\textup{(P)}$, there exist $\alpha\in [p^+,t^-)$ and $e\in L^1_+(\Omega)$ such that $$\alpha F(x,\tau)-\tau f(x,\tau)\leq\frac{\alpha-p^+}{p^+}w(x)|\tau|^{p(x)}+e(x)\  \text{for a.e.}\ x\in\Omega\ \text{and all}\ \tau\in\mathbb{R}.$$
	\end{itemize}
The existence result for the locally $p(\cdot)$-superlinear case is stated as follows. 
\begin{theorem}\label{Theo.Superlinear}
Let 
$\textup{(P)}, (\textup{A0})-(\textup{A4})$ and $\textup{(F3)}-\textup{(F6)}$ hold. For a given $\lambda\in (0,\mu_1]$, there exists a sequence $\{\theta_n\}_{n=1}^\infty$ with $0<\theta_{n+1}<\theta_n$ for all $n\in\mathbb{N}$ such that for any $\theta\in (\theta_{n+1},\theta_n)$, problem \eqref{Eq.General} admits at least $n$ distinct pairs of nontrivial solutions.
\end{theorem}

\par
In the last part, we study problem~\eqref{Eq.General} with the nonlinearity $f$ of locally sandwich-type growth. Precisely, we consider problem~\eqref{Eq.General} of the form:
	\begin{align}\label{Eq.Sandwich}
\begin{cases}
-\operatorname{div}a(x,\nabla u)=\lambda m(x)|u|^{s(x)-2}u+\theta b(x)|u|^{t(x)-2}u \quad \text{in } \Omega,\\
u=0\quad \text{on } \partial \Omega, 
\end{cases}
\end{align}
where $a$ satisfies $(\textup{A0})-(\textup{A5})$ and the following assumptions hold:
\begin{itemize}
	\item [$\textup{(S)}$] $s\in C_+(\overline{\Omega})$ with $s^+<p^-$ and $\left(\frac{t}{p}\right)^+<\left(\frac{t}{s}\right)^-$.
	\item [$\textup{(W)}$] $m\in L^{\frac{t(\cdot)}{t(\cdot)-s(\cdot)}}\left(b^{-\frac{s}{t-s}},\Omega\right)$.
	\end{itemize}	
Furthermore, the subcritical nonlinearity satisfies the following local condition:
\begin{itemize}
		\item [$\textup{(L1)}$] There exists a ball $B\subset\Omega$ such that $q_B<s_B^-$ and $\operatorname{meas}\, \{x\in B:\, m(x)>0 \}>0$, where $q_B$ is defined by $B$ as in $(\textup{A5})$.
\end{itemize}
Note that under $\textup{(L1)}$, from \eqref{A2'} and \eqref{A5'} we find positive constants $C_B$ and $C_p$ such that
\begin{equation}\label{A6}
	|A(x,\xi)|\leq C_B|\xi|^{q_B}+C_p|\xi|^{p_B^+}\quad \text{for a.e. } x\in B\ \text{and for all } \xi\in\R^N.
\end{equation}
Let us start to study the sandwich case by investigating the multiplicity of solutions to problem~\eqref{Eq.Sandwich}. The following multiple existence result is a counterpart of \cite[Theorem 1]{BBF.2021} for the bounded domain case involving Leray-Lions type operators.
\begin{theorem}\label{Theo.Sandwich-infty} 
	Let $(\textup{A0})-(\textup{A5})$, $\textup{(S)}$, $\textup{(W)}$ and $\textup{(L1)}$ hold. Assume furthermore that $m(x)>0$ for a.e. $x\in B$. Then there exists $\{\theta_k\}_{k=1}^\infty$ with $0<\theta_k<\theta_{k+1}$ for all $k\in\N$ such that for each $k\in\N$, for $\theta\in (0,\theta_k)$ given, there exist $\lambda_\star,\lambda^\star>0$  with $\lambda_\star<\lambda^\star$ such that for any $\lambda\in(\lambda_\star,\lambda^\star)$, problem \eqref{Eq.Sandwich} admits at least $k$ pairs of distinct solutions with negative energy. 
\end{theorem}
 It is worth pointing out that \cite[Theorem 1]{BBF.2021} actually asserted the existence of a sequence of solutions provided the $L^\infty$ norm of $K(x):=\theta b(x)$ is  sufficiently small. However, we found that their argument can only produce a finite number of solutions in the same way as the statement of Theorem~\ref{Theo.Sandwich-infty} and discussed it with the authors of \cite{BBF.2021} about this issue, please see Remark~\ref{compareBBF} for more details.

 
 \par
 Finally, we investigate a nontrivial nonnegative solution to problem~\eqref{Eq.Sandwich} changing a role of parameters. For this purpose, in addition to $\textup{(L1)}$ we assume that	
	 \begin{itemize}
		\item [$\textup{(L2)}$] There exists $\phi\in C_c^\infty(B)$ such that  $\int_B \frac{m(x)}{s(x)}\phi_+^{s(x)}\diff x>0$ with $\phi_+:= \max\{\phi,0\}$ and for each $\tau>0$, there exists $s_B(\tau)\in [s_B^-,s_B^+]$ such that
		$$\int_B \frac{m(x)}{s(x)}\phi_+^{s(x)}\tau^{s(x)}\diff x=\tau^{s_B(\tau)}\int_B \frac{m(x)}{s(x)}\phi_+^{s(x)}\diff x.$$ 
	\end{itemize} 
 Clearly, $\textup{(L2)}$ automatically holds true if either $m(x)>0$ a.e. in $B$ or  $s(\cdot)$ is constant in $B$ and $\operatorname{meas}\, \{x\in B:\, m(x)>0 \}>0$ (see e.g., \cite[Proposition 4.2]{Kawohl}).  	
Denote by $\mathcal{A}$ the set of all $\phi\in C_c^\infty(B)$ satisfying $\textup{(L2)}$ and define
\begin{equation}\label{lamda*}
\lambda_\star:=\inf_{\phi\in\mathcal{A}}\, \max_{\eta\in [s_B^-,s_B^+]}\, \frac{C^*(\eta)}{\int_B \frac{m(x)}{s(x)}\phi_+^{s(x)}\diff x}\left(\int_B |\nabla \phi|^{p_B^+}\diff x\right)^{\frac{\eta-q_B}{p_B^+-q_B}}\left(\int_B |\nabla \phi|^{q_B}\diff x\right)^{\frac{p_B^+-\eta}{p_B^+-q_B}},
\end{equation}
where
$$C^*(\eta):=\left(\frac{p_B^+-q_B}{p_B^+-\eta}\right)^{\frac{p_B^+-\eta}{p_B^+-q_B}}\left(\frac{p_B^+-q_B}{\eta-q_B}\right)^{\frac{\eta-q_B}{p_B^+-q_B}}C_p^{\frac{\eta-q_B}{p_B^+-q_B}}C_B^{\frac{p_B^+-\eta}{p_B^+-q_B}}\ \ \text{for}\ \ \eta\in [s_B^-,s_B^+].$$ 
Here, $C_p$ and $C_B$ are constants taken from \eqref{A6}. Clearly, $\lambda_\star$ is well defined and $\lambda_\star\in [0,\infty)$. The existence of a nontrivial nonnegative solution for the sandwich case is shown in the next theorem.

\begin{theorem}\label{Theo.Sandwich} Let $(\textup{A0})-(\textup{A5})$, $\textup{(S)}$, $\textup{(W)}$, $\textup{(L1)}$ and $\textup{(L2)}$ hold. Then for each given $\lambda>\lambda_\star,$ there exists $\theta_\star>0$ such that for any $\theta\in [0,\theta_\star)$, problem~\eqref{Eq.Sandwich} has a nontrivial nonnegative solution. 
\end{theorem}
\begin{remark}\rm
	It is worth pointing out that  if \ $b(x)|u|^{t(x)-2}u$\ is replaced by a subcritical term\ $\omega(x)|u|^{r(x)-2}u$, where $s^+<p^-\leq p^+<r^-$ and $r(x)<p^*(x)$ for all $x\in\overline{\Omega}$ and $\omega\in L_+^{\frac{p^\ast(\cdot)}{p^\ast(\cdot)-r(\cdot)}}(\Omega)$, then we can obtain one more nontrivial nonnegative solution to problem~\eqref{Eq.Sandwich} without assuming any combined condition on $m$ and $\omega$ by applying the Mountain Pass Theorem. 
	If $r(\cdot)$ is constant, it suffices to assume $\omega\in L^{\frac{p^\ast(\cdot)}{p^\ast(\cdot)-r(\cdot)}}(\Omega)$ and $\operatorname{meas}\, \{x\in \Omega:\, \omega(x)>0 \}>0$ instead of $\omega\in L_+^{\frac{p^\ast(\cdot)}{p^\ast(\cdot)-r(\cdot)}}(\Omega)$ (see Theorem~\ref{Theo.Sandwich.Subcritical}).
\end{remark}

\par
\noindent\textbf{Outline of the paper.} We organize the paper as follows: In Section~\ref{Pre} we review the definition and properties of Lebesgue-Sobolev spaces with variable exponents and introduce notations that are used throughout this paper. In Section~\ref{sub-super} we investigate the existence of infinitely many
solutions for the locally $p(\cdot)$-sublinear case and the multiplicity of solutions for the locally $p(\cdot)$-superlinear case by giving proofs of Theorems~\ref{Theo.Sublinear} and \ref{Theo.Superlinear}. Section~\ref{sandwich} is devoted to the study of the existence of $k$ pairs of solutions with an arbitrarily given number $k$ and a nontrivial nonnegative solution to problem \eqref{Eq.General} when the nonlinear $f$ is of a (locally) sandwich type by giving a proof of  Theorem~\ref{Theo.Sandwich-infty} and Theorem~\ref{Theo.Sandwich}. In this section, we also discuss problem~\eqref{Eq.Sandwich} when the growth is subcritical. Finally, in Section~\ref{Comments}, we comment on the regularity of solutions to problem~\eqref{Eq.General} obtained in the previous sections.


\section{Notions and preliminary results}\label{Pre}

Let $\Omega$ be a bounded Lipschitz domain in $\mathbb{R}^N.$ Denote
$$
C_+(\overline\Omega):=\left\{h\in C(\overline\Omega):
1<\inf_{x\in\overline\Omega}h(x)\leq \sup_{x\in\overline\Omega}h(x)<\infty\right\},
$$
and for $h\in C_+(\overline\Omega),$ denote
$$
h^+:=\sup\limits_{x\in \overline\Omega}h(x)\ \  \hbox{and}\ \
h^-:=\inf\limits_{x\in \overline\Omega}h(x).
$$
For $p\in C_+(\overline\Omega)$ and for a Lebesgue measurable
and positive a.e. function $w : \Omega\to \R$, define the weighted variable exponent Lebesgue space $L^{p(\cdot)}(w,\Omega)$ as
$$
L^{p(\cdot)}(w,\Omega) := \left \{ u : \Omega\to\mathbb{R}\  \hbox{is measurable},\ \int_\Omega w(x)|u(x)|^{p(x)}\diff x < \infty \right \}
$$
endowed with the Luxemburg norm
$$
\norm{u}_{L^{p(\cdot)}(w,\Omega)}:=\inf\left\{\lambda >0:
\int_\Omega
w(x)\left|\frac{u(x)}{\lambda}\right|^{p(x)}\diff x\le1\right\}.
$$
Set $L_+^{p(\cdot)}(w,\Omega):=\left\{u\in L^{p(\cdot)}(w,\Omega): \ u>0 \ \text{a.e. in}\ \Omega\right\}$. When $w\equiv 1$, we write  $L^{p(\cdot) }(\Omega) $, $L_+^{p(\cdot)}(\Omega)$  and $\norm{u}_{L^{p(\cdot)}(\Omega)}$  in place of $L^{p(\cdot)}(w,\Omega)$, $L_+^{p(\cdot)}(w,\Omega)$ and $\norm{u}_{L^{p(\cdot)}(w,\Omega)}$, respectively. Some basic properties of $L^{p(\cdot)}(w,\Omega)$ are listed in the next three propositions.

\begin{proposition} \label{Holder ineq}{\rm (\cite{FZ, KR})}
	The space $L^{p(\cdot)}(\Omega)$ is a separable, uniformly
	convex Banach space, and its dual space is
	$L^{p'(\cdot)}(\Omega),$ where $1/p(x)+1/p'(x)=1$. For any $u\in
	L^{p(\cdot)}(\Omega)$ and $v\in L^{p'(\cdot)}(\Omega)$, we have
	$$
	\Big|\int_\Omega uv\diff x\Big|
	\le
	2\norm{u}_{L^{p(\cdot)}(\Omega)}\norm{v}_{L^{p'(\cdot)}(\Omega)}.
	$$
\end{proposition}

\begin{proposition}[\cite{Diening}] \label{norm-modular}
	Define $\rho :L^{p(\cdot) }(w,\Omega )$ $ \to \mathbb{R}$ as
	\[
	\rho (u):=\int_{\Omega }w(x)| u(x)| ^{p(x) }\diff x,\quad
	\forall u\in L^{p(\cdot)}(w,\Omega ) .
	\]
	Then, for all $u\in L^{p(\cdot) }(w,\Omega)$  we have
	\begin{itemize}
		\item[(i)] $\norm{u}_{L^{p(\cdot)}(w,\Omega)}<1$ $(=1,>1)$
		if and only if \  $\rho (u) <1$ $(=1,>1)$, respectively;
		
		\item[(ii)] if \  $\norm{u}_{L^{p(\cdot)}(w,\Omega)}>1,$ then  $\norm{u}^{p^{-}}_{L^{p(\cdot)}(w,\Omega)}\leq \rho (u) \leq \norm{u}_{L^{p(\cdot)}(w,\Omega)}^{p^{+}}$;
		\item[(iii)] if \ $\norm{u}_{L^{p(\cdot)}(w,\Omega)}<1,$ then $\norm{u}_{L^{p(\cdot)}(w,\Omega)}^{p^{+}}\leq \rho
		(u) \leq \norm{u}_{L^{p(\cdot)}(w,\Omega)}^{p^{-}}$.
	\end{itemize}
Consequently, we have
$$\|u\|_{L^{p(\cdot)}(w,\Omega)}^{p^-}-1\leq \rho (u)\leq \|u\|_{L^{p(\cdot)}(w,\Omega)}^{p^+}+1, \quad \forall u\in L^{p(\cdot)}(w,\Omega).$$
	
\end{proposition}

\begin{proposition} [\cite{Kim-Sim}] \label{norm-norm}
	Let $p\in C_{+}(\overline{\Omega })$ and $q\in C(\overline{\Omega })$ such that $pq\in C_{+}(\overline{\Omega })$. For all $u\in L^{p(\cdot)q(\cdot) }(w,\Omega )$, it holds that
	\begin{itemize}
		
		\item[(i)] if  $\|u\|_{L^{p(\cdot)q(\cdot)}(w,\Omega)}\geq 1$, then  $\|u\|_{L^{p(\cdot)q(\cdot)}(w,\Omega)}^{q^-}\leq \big\||u|^q\big\|_{L^{p(\cdot)}(w,\Omega)}\leq \|u\|_{L^{p(\cdot)q(\cdot)}(w,\Omega)}^{q^+}$;
		\item[(ii)] if  $\|u\|_{L^{p(\cdot)q(\cdot)}(w,\Omega)}<1$, then  $\|u\|_{L^{p(\cdot)q(\cdot)}(w,\Omega)}^{q^+}\leq \big\||u|^q\big\|_{L^{p(\cdot)}(w,\Omega)}\leq \|u\|_{L^{p(\cdot)q(\cdot)}(w,\Omega)}^{q^-}$.
	\end{itemize}
	Consequently, we have
	$$\|u\|_{L^{p(\cdot)q(\cdot)}(w,\Omega)}^{q^-}-1\leq \big\||u|^q\big\|_{L^{p(\cdot)}(w,\Omega)}\leq \|u\|_{L^{p(\cdot)q(\cdot)}(w,\Omega)}^{q^+}+1,\quad \forall u\in L^{p(\cdot)q(\cdot) }(w,\Omega).$$
\end{proposition}

Define
\[
W^{1,p(\cdot)}(\Omega):=\{u\in L^{p(\cdot) }(\Omega) :
|\nabla u|\in L^{p(\cdot) }(\Omega ) \}
\]
endowed with the norm
\[
\|u\|_{W^{1,p(\cdot)}(\Omega )}:=\|u\|_{L^{p(\cdot)}(\Omega )}+\big\||\nabla u|\big\|_{L^{p(\cdot)}(\Omega )}.
\]
We recall the following crucial imbeddings on $W^{1,p(\cdot)}(\Omega)$. 
\begin{proposition}[\cite{KR,FZ,Diening}] \label{Compact.Imb} The space $W^{1,p(\cdot)}(\Omega )$ is a reflexive separable Banach space. Moreover, the following assertions hold:  
	\begin{itemize}
		\item [(i)] if $q\in C_+(\overline{\Omega})$ with $q(x) < p^{\ast }(x)$ for all $x\in\overline{\Omega}$, then
		$$W^{1,p(\cdot)}(\Omega)\hookrightarrow\hookrightarrow L^{q(\cdot)}(\Omega);$$
		\item [(ii)] if $p\in C^{0, \frac{1}{|\log t|}}(\overline{\Omega})$ and $p^+<N$, then
		$$W^{1,p(\cdot)}(\Omega) \hookrightarrow L^{p^\ast(\cdot) }(\Omega ).$$
	\end{itemize}
\end{proposition}
\noindent The space $W_0^{1,p(\cdot)}(\Omega )$ defined in Section~\ref{Intro} is in fact the closure of $C_c^\infty(\Omega)$ in $W^{1,p(\cdot)}(\Omega)$ due to the following Poincar\'{e} type inequality.
\begin{proposition}[\cite{FZ}]
	There exists a positive constant $C$ such that
	\begin{equation*}
		\|u\|_{L^{p(\cdot)}(\Omega)}\leq C\big\|\nabla u\big\|_{L^{p(\cdot)}(\Omega)},\quad \forall u\in W_0^{1,p(\cdot)}(\Omega).
	\end{equation*}
\end{proposition}

Next, we present a compact result that will be frequently utilized in the next sections. Let $t\in C_+(\overline\Omega)$ and $b:\Omega\to\R$ be measurable and positive a.e. in $\Omega$ such that $W_0^{1,p(\cdot)}(\Omega)\hookrightarrow L^{t(\cdot)}(b,\Omega)$. Let $m\in L^{\frac{t(\cdot)}{t(\cdot)-\sigma(\cdot)}}\big(b^{-\frac{\sigma}{t-\sigma}},\Omega\big)$ for some $\sigma\in C(\overline{\Omega })$ satisfying $1\leq \sigma (x)<t(x)$ for all $x\in \overline{\Omega }.$ Then, $m|u|^{\sigma}\in L^1(\Omega)$ for all $u\in W_0^{1,p(\cdot)}(\Omega)$ in view of Proposition~\ref{Holder ineq}. Furthermore, we have the following. 

\begin{proposition}\label{le.compact.imb}
	 Let $t,b$ and $\sigma$ be as above. Then, for $u_n \rightharpoonup  u$ in $W_0^{1,p(\cdot)}(\Omega)$ as $n\to\infty$, we have
	\begin{equation}\label{le.compact.imb.lim1}
		\int_\Omega m(x)|u_n|^{\sigma(x)}\diff x\to \int_\Omega m(x)|u|^{\sigma(x)}\diff x\ \ \text{as}\ n\to\infty,
	\end{equation}
	\begin{equation}\label{le.compact.imb.lim2}
		\int_\Omega |m(x)||u_n-u|^{\sigma (x)}\diff x\to 0\ \ \text{as}\ n\to\infty
	\end{equation}
and
\begin{equation}\label{le.compact.imb.lim3}
	\int_\Omega |m(x)||u_n|^{\sigma (x)-1}|u_n-u|\diff x\to 0\ \ \text{as}\ n\to\infty.
\end{equation}
Consequently, if $m\in L_+^{\frac{t(\cdot)}{t(\cdot)-\sigma(\cdot)}}\big(b^{-\frac{\sigma}{t-\sigma}},\Omega\big)$, then
$$W_0^{1,p(\cdot)}(\Omega)\hookrightarrow\hookrightarrow L^{\sigma(\cdot)}(m,\Omega).$$
\end{proposition}

\begin{proof}  Clearly, \eqref{le.compact.imb.lim1} follows if we can prove that
	\begin{equation}\label{PL.compact.imb.lim1}
		\int_\Omega |v_n| \diff x\to 0\ \ \text{as}\ n\to\infty,
	\end{equation}
	where $v_n:=m\left(|u_n|^{\sigma(x)}-|u|^{\sigma(x)}\right).$ To this end, we first note that up to a subsequence, $v_n\to 0$ a.e. in $\Omega$ as $n\to\infty.$ Then, by the Vitali convergence theorem it suffices to show that for a given and arbitrary $\epsilon>0$,  there exists $\delta=\delta(\epsilon)>0$ such that for any measurable subset $Q\subset\Omega$ with $\operatorname{meas} (Q)<\delta$,
	\begin{equation}\label{PL.v_n}
		\int_{Q}|v_n|\diff x <\epsilon.
	\end{equation}
	Indeed, let $\epsilon>0$ be given and arbitrary. Since $\{u_n\}_{n=1}^\infty$ is bounded in $W_0^{1,p(\cdot)}(\Omega)$ we have that $\{u_n\}_{n=1}^\infty$ is bounded in $L^{t(\cdot)}(b,\Omega)$ and hence
	\begin{equation}\label{PL.compact.defM}
		M:=\epsilon+2\left[2+\|u\|^{\sigma^+}_{L^{t(\cdot)}(b,\Omega)}+\sup_{n\in\N}\|u_n\|^{\sigma^+}_{L^{t(\cdot)}(b,\Omega)}\right]\in (0,\infty).
	\end{equation}
	On the other hand, the assumption $m\in L^{\frac{t(\cdot)}{t(\cdot)-\sigma(\cdot)}}\big(b^{-\frac{\sigma}{t-\sigma}},\Omega\big)$ implies that there exist $\delta=\delta(\epsilon)>0$ such that for any measurable subset $Q\subset\Omega$ with $\operatorname{meas} (Q)<\delta$, 
	\begin{equation}\label{PL.compact.delta}
		\int_{Q}\left|m(x)b(x)^{-\frac{\sigma(x)}{t(x)}}\right|^{\frac{t(x)}{t(x)-\sigma(x)}}\diff x<\left(\frac{\epsilon}{M}\right)^{\left(\frac{t}{t-\sigma}\right)^+}(<1).
	\end{equation}
	Invoking Proposition~\ref{norm-modular} and using \eqref{PL.compact.defM}-\eqref{PL.compact.delta} we have
	\begin{align*}
		\int_{Q}|v_n|\diff x& =\int_{Q} |m(x)|\big||u_n|^{\sigma(x)}-|u|^{\sigma(x)}\big|\diff x\\
		&\leq 2\|mb^{-\frac{\sigma}{t}}\|_{L^{\frac{t(\cdot)}{t(\cdot)-\sigma(\cdot)}}(Q)}\big\|b||u_n|^{\sigma}-|u|^{\sigma}|\big\|_{L^{\frac{t(\cdot)}{\sigma(\cdot)}}(Q)}\notag\\
		&\leq 2\|mb^{-\frac{\sigma}{t}}\|_{L^{\frac{t(\cdot)}{t(\cdot)-\sigma(\cdot)}}(Q)}\left[2+\|u_n\|^{\sigma^+}_{L^{t(\cdot)}(b,Q)}+\|u\|^{\sigma^+}_{L^{t(\cdot)}(b,Q)}\right]\notag\\
		&\leq M\left(\int_{Q}\left|m(x)b(x)^{-\frac{\sigma(x)}{t(x)}}\right|^{\frac{t(x)}{t(x)-\sigma(x)}}\diff x\right)^{\frac{1}{\left(\frac{t}{t-\sigma}\right)^+}}\notag\\
		&<\epsilon,\ \forall n\in\mathbb{N}.  
	\end{align*}
	That is, we have proved \eqref{PL.v_n}, and hence, in turn we obtain \eqref{PL.compact.imb.lim1} and  \eqref{le.compact.imb.lim1}. The proof for \eqref{le.compact.imb.lim2} is similar and we omit it. To see \eqref{le.compact.imb.lim3}, by Propositions~\ref{Holder ineq} and \ref{norm-modular} we have
	\begin{align*}
	\int_\Omega &|m(x)||u_n|^{\sigma (x)-1}|u_n-u|\diff x\\
	&\leq 2 \left\||m|^{\frac{\sigma-1}{\sigma}}|u_n|^{\sigma-1}\right\|_{L^{\frac{\sigma(\cdot)}{\sigma(\cdot)-1}}(\Omega)}\left\||m|^{\frac{1}{\sigma}}|u_n-u|\right\|_{L^{\sigma(\cdot)}(\Omega)}\\
	&\leq 2 \left[1+\left(\int_\Omega |m(x)||u_n|^{\sigma(x)}\diff x\right)^{\left(\frac{\sigma-1}{\sigma}\right)^+}\right]\times\\
	&\quad\quad\quad\times \max\left\{\left(\int_\Omega |m(x)||u_n-u|^{\sigma(x)}\diff x\right)^{\frac{1}{\sigma^+}},\left(\int_\Omega |m(x)||u_n-u|^{\sigma(x)}\diff x\right)^{\frac{1}{\sigma^-}}\right\}.
	\end{align*}
Combining this with \eqref{le.compact.imb.lim1} and \eqref{le.compact.imb.lim2} we obtain \eqref{le.compact.imb.lim3}.
\end{proof}

We conclude this section with a concentration-compactness principle for variable exponent spaces. Let $p\in C^{0, \frac{1}{|\log t|}}(\overline{\Omega})$ with $p^+<N$  and let $t\in C_+(\overline{\Omega})$ with $p(x)<t(x)\leq p^\ast(x)$ for all $x\in\overline{\Omega}$ such that 
$$\mathcal{C}:=\{x\in\overline{\Omega}:\, t(x)=p^\ast(x)\}\ne\emptyset.$$
Let $b\in L^\infty_+(\Omega)$. Then, it follows from \eqref{cri.imb} that
\begin{equation}\label{Sb}
	S_b:=\inf_{u\in W_0^{1,p(\cdot)}(\Omega)\setminus\{0\} }\frac{\|u\|}{\norm{u}_{L^{t(\cdot)}(b,\Omega)}}\in (0,\infty).
\end{equation}
Let $\mathcal{M}(\overline{\Omega})$ denote the space of Radon measures on $\overline{\Omega}$; that is, the dual of $C(\overline{\Omega})$. By the Riesz representation theorem, for each $\mu\in \mathcal{M}(\overline{\Omega}),$ there is a unique signed Borel measure on $\overline{\Omega}$, still denoted by $\mu$, such that
$$\langle \mu,f\rangle=\int_{\overline{\Omega}}f\diff\mu,\quad \forall f\in C(\overline{\Omega}).$$
We identify $L^1(\Omega)$ with a subspace of $\mathcal{M}(\overline{\Omega})$ through the imbedding
$T:\ L^1(\Omega)\to \mathcal{M}(\overline{\Omega})$  defined by 
$$\langle Tu,f\rangle=\int_{\Omega}uf\diff x, \ \ \forall u\in L^1(\Omega), \ \forall f\in C(\overline{\Omega})$$
(see e.g., \cite[p. 116]{Brezis-book}). The next theorem is the concentration-compactness principle for variable exponent spaces, which plays a key role in our arguments in the next sections.

\begin{theorem}[cf. \cite{Bonder,HS.NA2016}]\label{T.CCP}
	Let $\{u_n\}_{n=1}^\infty$ be a bounded sequence in
	$W_0^{1,p(\cdot)}(\Omega)$  such that
	\begin{eqnarray*}
		u_n &\rightharpoonup& u \quad \text{in}\quad  W_0^{1,p(\cdot)}(\Omega),\\
		|\nabla u_n|^{p(x)} &\overset{\ast }{\rightharpoonup }&\mu\quad \text{in}\quad \mathcal{M}(\overline{\Omega}),\\
		b|u_n|^{t(x)}&\overset{\ast }{\rightharpoonup }&\nu\quad \text{in}\quad \mathcal{M}(\overline{\Omega}).
	\end{eqnarray*}
	Then, there exist $\{x_i\}_{i\in \mathcal{I}}\subset \mathcal{C}$ of distinct points and $\{\nu_i\}_{i\in \mathcal{I}}, \{\mu_i\}_{i\in \mathcal{I}}\subset (0,\infty),$ where $\mathcal{I}$ is at most countable, such that
	\begin{gather*}
		\nu=b|u|^{t(x)} + \sum_{i\in \mathcal{I}}\nu_i\delta_{x_i},\label{CCP.form.nu}\\
		\mu \geq |\nabla u|^{p(x)} + \sum_{i\in \mathcal{I}} \mu_i \delta_{x_i},\label{CCP.form.mu}\\
		S_b \nu_i^{1/t(x_i)} \leq \mu_i^{1/p(x_i)}, \quad \forall i\in \mathcal{I}.\label{CCP.nu_mu}
	\end{gather*}
\end{theorem}

\vspace{0.3cm}
\noindent\textbf{Other notations:} 
\begin{itemize}
	\item [{}] $E$: the space $W_0^{1,p(\cdot)}(\Omega)$ endowed with the norm $\|u\|:=\|\nabla u\|_{L^{p(\cdot)}(\Omega)}$
	\item [{}]$E^*$: the dual of $E$
	\item[{}] $\langle \cdot,\cdot\rangle$: the duality pairing between a Banach space (from the context) and its dual
	\item[{}] $B_\rho$: the open ball in $E$ centered at $0$ with radius $\rho$.
	\item[{}] $B_\rho(x_0)$: the open ball in $\R^N$ centered at $x_0$ with radius $\rho$
	\item[{}] $|Q|$: the Lebesgue measure of $Q\subset\R^N$
\end{itemize} 

\section{The locally $p(\cdot)$-sublinear and locally $p(\cdot)$-superlinear cases}\label{sub-super}
In this section, we will prove Theorems~\ref{Theo.Sublinear} and \ref{Theo.Superlinear} by employing variants of the Mountain Pass Theorem.


\subsection{The locally $p(\cdot)$-sublinear case}

The existence of a sequence of small solutions to problem~\eqref{Eq.General} for the locally $p(\cdot)$-sublinear case, Theorem~\ref{Theo.Sublinear}, can be obtained easily by adapting \cite[Proof of Theorem 1.1]{HHS}, which employed a critical theorem obtained in \cite{Kaj.2006}.
\begin{proof} [\textbf{Proof of Theorem~\ref{Theo.Sublinear}}] Let $\lambda>0$ and $\theta\geq 0$. Consider the problem
	\begin{eqnarray*}
		\begin{cases}
			-\operatorname{div}a(x,\nabla u)= \widetilde{f}(x,u) \quad &\text{in } \Omega ,\\
			u=0\quad &\text{on } \partial \Omega,
		\end{cases}
	\end{eqnarray*}
where $\widetilde{f}(x,u):=\lambda f(x,u)+\theta b(x)|u|^{t(x)-2}u$. Note that $\widetilde{f}$ also satisfies $(\textup{F}1)$ and $(\textup{F}2)$ in place of $f$ and hence, by repeating the argument used in \cite[Proof of Theorem 1.1]{HHS} in which we apply \cite[Theorem 1.2]{HNT} instead of \cite[Proposition A.1]{HHS}, the desired conclusion follows.
	\end{proof}
\begin{remark}\rm It is clear that Theorem~\ref{Theo.Sublinear} remains valid for any $t\in C_+(\overline{\Omega})$ and therefore, we do not need to assume $p\in C^{0, \frac{1}{|\log t|}}(\overline{\Omega})$ when $p^+<N$ in this theorem.
\end{remark}

\subsection{The locally $p(\cdot)$-superlinear case}\label{Subsec3.2}

In this subsection, we will prove Theorem~\ref{Theo.Superlinear} by using the idea of \cite{Sil-Xav.2003}. We will restrict our proof to the case $p^+<N$ since the proof for the case $p^+\geq N$ is similar and simpler. We will make use of the following abstract result for symmetric $C^1$ functionals.
\begin{proposition}[\cite{Sil-Xav.2003}]\label{prop.abs}
	Let $X = V \oplus W$, where $X$ is a real Banach space and $V$ is finite dimensional.
	Suppose that $I \in C^1(X,\mathbb{R})$ is an even functional satisfying $I(0) = 0$ and
	\begin{itemize}
		\item [$(\textup{I}1)$] there exist constants $\rho,\ \beta > 0$ such that $I(u)\geq \beta$ for all $u\in\partial B_\rho\cap W;$
		\item [$(\textup{I}2)$] there exist a subspace $\widetilde{X}$ of $X$ and $L>0$ such that $\operatorname{dim} V < \operatorname{dim} \widetilde{X}<\infty$ and $\underset{u\in \widetilde{X}}{\sup}\, I(u)<L$;
		\item [$(\textup{I}3)$] considering $L$ given by $(\textup{I}2)$, $I$ satisfies the $\textup{(PS)}_c$ condition for any  $c\in [0,L]$.
	\end{itemize}
	Then $I$ possesses at least $\operatorname{dim} \widetilde{X}-\operatorname{dim} V$ pairs of nontrivial critical points.
\end{proposition}

To determine solutions to problem \eqref{Eq.General}, we will apply Proposition~\ref{prop.abs} for $X= E$ and  $I:\, E \to\mathbb{R}$ defined as 
\begin{equation}\label{def.I}
	I(u):=\int_\Omega A(x,\nabla u)\diff x-\lambda\int_\Omega F(x,u)\diff x-\theta\int_\Omega \frac{b(x)}{t(x)}|u|^{t(x)}\diff x,\ u\in E.
\end{equation}
It is clear that  under the assumption of Theorem~\ref{Theo.Superlinear}, $I$ is of class $C^1$ and  a critical point of $I$ is a weak solution to problem \eqref{Eq.General}. Moreover, $I$ is even on $E$ and $I(0)=0.$ 
In order to verify condition $(\textup{I}3)$ of Proposition~\ref{prop.abs}, we need the following lemma. 
\begin{lemma} \label{le.ps1}
	Let $\textup{(P)}$, $\textup{(F4)}$ and $\textup{(F6)}$ hold. Then, for $\lambda\in (0,\mu_1]$ and $\theta>0$, $I$ satisfies the $\textup{(PS)}_c$ condition for all $c\in\mathbb{R}$ satisfying
		\begin{equation}\label{PS_c1}
		c<-\frac{\|e\|_{L^1(\Omega)}\lambda}{\alpha}+\left(\frac{1}{\alpha}-\frac{1}{t^-}\right)S_b^N\min\left\{\theta^{-\frac{1}{h^+-1}},\theta^{-\frac{1}{h^--1}}\right\}
		\end{equation}
		with $e$, $\alpha$ given by $\textup{(F6)}$, $S_b$ given by \eqref{Sb} and $h(x):=\frac{t(x)}{p(x)}$ for $x\in\overline{\Omega }$.
		
\end{lemma}
\begin{proof} Let $\lambda\in (0,\mu_1]$ and $\theta>0$ and let $\{u_n\}_{n=1}^\infty$ be a $\textup{(PS)}_c$-sequence for $I$ with $c$ satisfying \eqref{PS_c1}. We first claim that $\{u_n\}_{n=1}^\infty$ is bounded in $E$. Indeed, by $(\textup{A3})$ we have
	\begin{multline*}
		I(u_n)-\frac{1}{\alpha}\langle I'(u_n) ,u_n \rangle \geq  \left(\frac{1}{p^+}-\frac{1}{\alpha}\right)\int_{\Omega}a(x,\nabla u_n)\cdot \nabla u_n\diff x \\
		-\frac{\lambda}{\alpha}\int_{\Omega}\left[\alpha F(x,u_n)-f(x,u_n)u_n\right]\diff x+\theta\int_{\Omega}\left(\frac{1}{\alpha}-\frac{1}{t(x)}\right)b(x)|u_n|^{t(x)}\diff x.
	\end{multline*}
From this, $(\textup{A3})$ and $\textup{(F6)}$ we obtain
\begin{multline}\label{PT.super.*}
I(u_n)-\frac{1}{\alpha}\langle I'(u_n) ,u_n \rangle \geq  \left(\frac{1}{p^+}-\frac{1}{\alpha}\right)\int_{\Omega}|\nabla u_n|^{p(x)}\diff x 
-\frac{\lambda (\alpha-p^+)}{\alpha p^+}\int_{\Omega}w(x)|u_n|^{p(x)}\diff x\\
-\frac{\lambda\|e\|_{L^1(\Omega)}}{\alpha}+\theta\left(\frac{1}{\alpha}-\frac{1}{t^-}\right)\int_{\Omega}b(x)|u_n|^{t(x)}\diff x.
\end{multline}
Note that by $\textup{(P)}$ it holds that
\begin{equation}\label{Poincare.ineq}
	\frac{\left(\alpha-p^+\right)\mu_1}{p^+}\int_\Omega w(x)|u_n|^{p(x)}\diff x \leq \frac{\alpha-p^+}{p^+}\int_\Omega|\nabla u_n|^{p(x)}\diff x, \quad \forall n\in\N.
\end{equation}
Combining \eqref{PT.super.*} with \eqref{Poincare.ineq} and recalling $I(u_n)\to c$ and $I'(u_n)\to 0$ we deduce that for $n$ large,
\begin{multline*}
	c+1+\|u_n\|\geq \frac{\left(\mu_1-\lambda\right) (\alpha-p^+)}{\mu_1\alpha p^+}\int_{\Omega}|\nabla u_n|^{p(x)}
	\,\diff x-\frac{\lambda \|e\|_{L^1(\Omega)}}{\alpha}\\
	+\theta\left(\frac{1}{\alpha}-\frac{1}{t^-}\right)\int_{\Omega}b(x)|u_n|^{t(x)}\diff x.
\end{multline*}
Thus,
\begin{equation}\label{PT.super.boundedness1}
\int_{\Omega}b(x)|u_n|^{t(x)}\diff x \leq C_1(1+\|u_n\|),\ \ \forall n\in\mathbb{N}.
\end{equation}		
Here and in the remaining proof, $C_i$ ($i\in\mathbb{N}$) denotes a positive constant independent of $n$. On the other hand, we deduce from $(\textup{A3})$,  $\textup{(F4)}$ and Proposition~\ref{norm-modular} that for $n$ large,
\begin{align}\label{PT.super.boundedness2}
\notag\frac{1}{p^+}\left(\|u_n\|^{p^-}-1\right)&\leq I(u_n)+\lambda\int_\Omega F(x,u_n)\diff x+\theta\int_{\Omega}\frac{b(x)}{t(x)}|u_n|^{t(x)}\diff x\\
&\leq c+1+\lambda\sum_{j=1}^{m_0}\int_{\Omega}\frac{a_j(x)}{r_j(x)}|u_n|^{r_j(x)}\diff x+\frac{\theta}{t^-}\int_{\Omega}b(x)|u_n|^{t(x)}\diff x.
\end{align}
By the Young inequality, we have
$$a_j|u_n|^{r_j(x)}\leq \frac{t(x)-r_j(x)}{t(x)}\left|a_jb^{-\frac{r_j(x)}{t(x)}}\right|^{\frac{t(x)}{t(x)-r_j(x)}}+\frac{r_j(x)}{t(x)}b|u_n|^{t(x)}.$$
Plugging this into \eqref{PT.super.boundedness2} and recalling the assumption of $a_j$, we arrive at
\begin{equation}\label{PT.PS1.un-boundedness.2}
\|u_n\|^{p^-}\leq C_2\left(1+\int_{\Omega}b(x)|u_n|^{t(x)}\diff x\right).
\end{equation}
From \eqref{PT.super.boundedness1} and \eqref{PT.PS1.un-boundedness.2} we obtain the boundedness of $\{u_n\}_{n=1}^\infty$ in $E$ since $p^->1$. Then by Theorem~\ref{T.CCP}, there exist $\mu,\nu\in \mathcal{M}(\overline{\Omega})$, 
 $\{x_i\}_{i\in\mathcal{I}}\subset\mathcal{C}$, and $\{\mu_i\}_{i\in\mathcal{I}}, \{\nu_i\}_{i\in\mathcal{I}}\subset (0,\infty)$, where $\mathcal{I}$ is at most countable,  such that up to a subsequence we have
\begin{eqnarray}
u_n(x) &\to& u(x) \ \ \text{a.e.} \ \ x\in\Omega,\label{PT.super.a.e1}\\
u_n &\rightharpoonup& u \ \ \text{in} \  E,\label{PT.super.weak1}\\
|\nabla u_n|^{p(x)} &\overset{\ast }{\rightharpoonup }&\mu \geq |\nabla u|^{p(x)} + \sum_{i\in \mathcal{I}} \mu_i \delta_{x_i} \ \text{in}\  \mathcal{M}(\overline{\Omega}),\label{PT.super.mu1}\\
b|u_n|^{t(x)}&\overset{\ast }{\rightharpoonup }&\nu=b|u|^{t(x)} + \sum_{i\in \mathcal{I}}\nu_i\delta_{x_i} \ \text{in}\ \mathcal{M}(\overline{\Omega}),\label{PT.super.nu1}\\
S_b \nu_i^{1/t(x_i)} &\leq& \mu_i^{1/p(x_i)}, \ \ \forall i\in \mathcal{I},\label{PT.super.mu-nu1}
\end{eqnarray}
where $\delta_{x_i}$ is the Dirac mass at $x_i$. We claim that $\mathcal{I}=\emptyset.$ Suppose, on the contrary, that there exists  $i\in \mathcal{I}.$ Arguing similarly to that obtained \eqref{PT.super.boundedness1}, we have
	\begin{align*}
c+o_n(1)=I(u_n)-\frac{1}{\alpha}\langle I'(u_n) ,u_n \rangle\geq -\frac{\|e\|_{L^1(\Omega)}\lambda}{\alpha}+\left(\frac{1}{\alpha}-\frac{1}{t^-}\right)\theta\int_{\Omega}b(x)|u_n|^{t(x)}\diff x.
\end{align*}
Passing to the limit as $n\to\infty$ in the last inequality and using \eqref{PT.super.nu1}, we obtain
\begin{equation}\label{PT.super.ps.est.c}
c\geq -\frac{\|e\|_{L^1(\Omega)}\lambda}{\alpha}+ \left(\frac{1}{\alpha}-\frac{1}{t^-}\right)\theta\nu_i.
\end{equation}
Next, we will prove that
\begin{equation}\label{PL.PS.Super.est.nu-mu}
\mu_i\leq \theta\nu_i.
\end{equation}
To this end, let $\phi\in C_c^\infty(\mathbb{R}^N)$ be such that $0\leq \phi\leq 1,$ $\phi\equiv 1$ on $B_{\frac{1}{2}}(0)$ and $\operatorname{supp}\, (\phi)\subset B_1(0)$. For $\epsilon>0,$ define $\phi_{i,\epsilon}(x):=
\phi(\frac{x-x_i}{\epsilon})$ for $x\in\mathbb{R}^N$. Fixing such an $\epsilon,$ we get that
\begin{multline*}
\int_{\Omega}a(x,\nabla u_n)\cdot\nabla (\phi_{i,\epsilon}u_n)\diff x=\langle I'(u_n) ,\phi_{i,\epsilon}u_n \rangle+\lambda\int_{\Omega}f(x,u_n)\phi_{i,\epsilon}u_n\diff x\\
+\theta\int_{\Omega}b(x)|u_n|^{t(x)}\phi_{i,\epsilon}\diff x,
\end{multline*}
i.e.,
\begin{multline}\label{PL.PS.3.14}
\int_{\Omega}\phi_{i,\epsilon}a(x,\nabla u_n)\cdot\nabla u_n\diff x=\langle I'(u_n) ,\phi_{i,\epsilon}u_n \rangle-\int_{\Omega}u_na(x,\nabla u_n)\cdot\nabla \phi_{i,\epsilon}\diff x\\
+\lambda\int_{\Omega}\phi_{i,\epsilon}f(x,u_n)u_n\diff x+\theta\int_{\Omega}\phi_{i,\epsilon}b(x)|u_n|^{t(x)}\diff x.
\end{multline}
Let $\delta>0$ be arbitrary and fixed. By the Young inequality, we have 
\begin{align*}
|u_na(x,\nabla u_n)\cdot\nabla \phi_{i,\epsilon}|\leq\, &\delta|a(x,\nabla u_n)|^{\frac{p(x)}{p(x)-1}}+C(p,\delta)|\nabla \phi_{i,\epsilon}|^{p(x)}|u_n|^{p(x)},
\end{align*}
where  $C(p,\delta)$ is a positive constant independent of $n$ and $\epsilon$. Using this estimation and $(\textup{A3})$,  we deduce from \eqref{PL.PS.3.14} that
\begin{multline*}
	\int_{\Omega}\phi_{i,\epsilon}|\nabla u_n|^{p(x)}\diff x\leq \langle I'(u_n) ,\phi_{i,\epsilon}u_n \rangle+C_*\delta+C(p,\delta)\int_\Omega|\nabla \phi_{i,\epsilon}|^{p(x)}|u_n|^{p(x)}\diff x\\
	+ \lambda\int_{\Omega}\phi_{i,\epsilon}|f(x,u_n)u_n|\diff x+\theta\int_{\Omega}\phi_{i,\epsilon}b(x)|u_n|^{t(x)}\diff x,
\end{multline*}
where $C_*:=\sup_{n\in\N}\int_\Omega|a(x,\nabla u_n)|^{\frac{p(x)}{p(x)-1}}\diff x\in [0,\infty)$ due to \eqref{A2''} and the boundedness of $\{u_n\}_{n=1}^\infty$ in $E$. Then, by passing to the limit as $n\to\infty$ in the last estimation and invoking Propositions~\ref{Compact.Imb} and \ref{le.compact.imb} with taking into account  \eqref{PT.super.weak1}, the fact that $I'(u_n)\to 0$ and the boundedness of $\{\phi_{i,\epsilon}u_n\}_{n=1}^\infty$ in $E$, we arrive at
\begin{multline}\label{PL.PS.Super.lim}
	\int_{\overline{\Omega}}\phi_{i,\epsilon}\diff\mu\leq \, C_*\delta+ C(p,\delta)\int_\Omega|\nabla \phi_{i,\epsilon}|^{p(x)}|u|^{p(x)}\diff x\\
	+\lambda\int_{\Omega}\phi_{i,\epsilon}\sum_{j=1}^{m_0}a_j(x)|u|^{r_j(x)}\diff x+\theta\int_{\overline{\Omega}}\phi_{i,\epsilon}\diff\nu.
\end{multline}
Note that 
$u\in L^{p^\ast(\cdot)}(\Omega)$ in view of Proposition~\ref{Compact.Imb}; hence, by Proposition~\ref{Holder ineq}, 
\begin{align}\label{PL.PS.Super.est}
\int_\Omega|\nabla \phi_{i,\epsilon}|^{p(x)}|u|^{p(x)}\diff x \leq 2\big\| |u|^{p}\big\|_{L^{\frac{p^\ast(\cdot)}{p(\cdot)}}(B_\epsilon(x_i)\cap \Omega)}
\big\||\nabla \phi_{i,\epsilon}|^{p} \big\|_{L^{\frac{N}{p(\cdot)}}
	(B_\epsilon(x_i)\cap \Omega)}.
\end{align}
On the other hand, invoking Proposition~\ref{norm-modular} we have 
$$
\big\||\nabla \phi_{i,\epsilon}|^{p} \big\|_{L^{\frac{N}{p(\cdot)}}
	(B_\epsilon(x_i)\cap \Omega)}^{p^-}
\le 1+ \int_{B_\epsilon(x_i)\cap \Omega} |\nabla \phi_{i,\epsilon}|^N \diff x  \leq 1+\int_{B_1(0)} |\nabla \phi (y)|^N\, \diff y.
$$
From this and the fact that $u\in L^{p^\ast(\cdot)}(\Omega)$ it follows from \eqref{PL.PS.Super.est} that
$$\lim_{\epsilon\to 0^+}\int_\Omega|\nabla \phi_{i,\epsilon}|^{p(x)}|u|^{p(x)}\diff x=0.$$ Clearly, $\lim_{\epsilon\to 0^+}\int_{\Omega}\phi_{i,\epsilon}\sum_{j=1}^{m_0}a_j(x)|u|^{r_j(x)}\diff x=0$. By letting $\epsilon\to 0^+$ in \eqref{PL.PS.Super.lim} and taking into account the last two limits we arrive at
\begin{equation*}
\mu_i\leq C_*\delta+\theta\nu_i.
\end{equation*}
Thus, we obtain \eqref{PL.PS.Super.est.nu-mu} since $\delta>0$ was chosen arbitrarily. Combining \eqref{PL.PS.Super.est.nu-mu} and \eqref{PT.super.mu-nu1} and noting $t(x_i) = \frac{Np(x_i)}{N-p(x_i)}$, we have
\begin{equation}\label{PT.super.ps.mu-nu}
\theta \nu_i\geq \mu_i\geq S_b^N\theta^{-\frac{p(x_i)}{t(x_i)-p(x_i)}}.
\end{equation}	
Combining \eqref{PT.super.ps.est.c} and \eqref{PT.super.ps.mu-nu} gives
$$c\geq 	-\frac{\|e\|_{L^1(\Omega)}\lambda}{\alpha}+\left(\frac{1}{\alpha}-\frac{1}{t^-}\right)S_b^N\min\left\{\theta^{-\frac{1}{h^+-1}},\theta^{-\frac{1}{h^--1}}\right\},$$
which contradicts to \eqref{PS_c1}. That is, $\mathcal{I}=\emptyset$ and hence,
$$\int_{\Omega}b(x)|u_n|^{t(x)}\diff x\to \int_{\Omega}b(x)|u|^{t(x)}\diff x.$$
From this and \eqref{PT.super.a.e1} we deduce that $u_n\to u$ in $L^{t(\cdot)}(b,\Omega)$ in view of the Br\'ezis-Lieb lemma (see e.g., \cite[Lemma 3.6]{HS.NA2016}). Thus, by Propositions~\ref{Holder ineq} and \ref{norm-modular} we easily obtain
\begin{equation}\label{PT.super.ps.int.b}
\int_{\Omega}b(x)|u_n|^{t(x)-2}u_n(u_n-u)\diff x\to 0.
\end{equation}
On the other hand, in view of $\textup{(F4)}$ 
 and Proposition~\ref{le.compact.imb} it follows from \eqref{PT.super.weak1} that
 \begin{equation}\label{PT.super.ps.int.f}
 \int_\Omega f(x,u_n)(u_n-u)\diff x\to 0.
 \end{equation}
 Using \eqref{PT.super.ps.int.b}, \eqref{PT.super.ps.int.f}, the boundedness of $\{u_n\}_{n=1}^\infty$ in $E$ and the fact that $I'(u_n)\to 0$ in $E^*,$  we obtain
 \begin{multline*}
 	\int_{\Omega}a(x,\nabla u_n)\cdot (\nabla u_n-\nabla u)\diff x=\big\langle I'(u_n) ,u_n-u \big\rangle+\lambda\int_{\Omega}f(x,u_n)(u_n-u)\diff x\\
 	+\theta\int_{\Omega}b(x)|u_n|^{t(x)-2}u_n(u_n-u)\diff x\to 0;
 \end{multline*}
hence, $u_n\to u$ in $E$ due to the $(S_+)$ property of Leray-Lions operators with variable growths (see \cite[Theorem 4.1]{LKV.2009}). In other words, $I$ satisfies the $\textup{(PS)}_c$ condition and the proof is complete.
		\end{proof}
In order to verify condition $(\textup{I}2)$ of Proposition~\ref{prop.abs}, we will take $\widetilde{X}$ from a sequence $\{E_k\}_{k=1}^\infty$ of linear subspaces of $E$, which is constructed as follows. For each $k\in\mathbb{N}$, define
\begin{equation}\label{Ek}
	E_k:=\operatorname{span}\{\varphi_1,\varphi_2,\cdots,\varphi_k\},
\end{equation}	
	where $\varphi_n$ is an  eigenfunction corresponding to the $n^{\text{th}}$ eigenvalue of the following eigenvalue problem with $B$ in $\textup{(F5)}$: 
	\begin{equation*}\label{EP}
		\begin{cases}
			-\Delta u=\mu u \quad &\text{in } B,\\
			u=0 \quad &\text{on } \partial B,
		\end{cases}
	\end{equation*}
	which is extended to  $\Omega$ by putting $\varphi_n(x)=0$ for $x\in\Omega\setminus B.$

\begin{lemma}\label{PL.Super.Rk}
Let $\textup{(F4)}$ and $\textup{(F5)}$ hold.  Let $\lambda>0$ and let $\theta> 0$. Then, for each $k\in\mathbb{N}$, there exists $R_k>1$ independent of $\theta$ such that  
\begin{equation}\label{PT.Super1.Rk}
I(u)< 0,\ \ \forall u\in E_k\setminus B_{R_k}.
\end{equation}
\end{lemma}	
\begin{proof}
Let $k\in\mathbb{N}$. Since all norms on $E_k$ are mutually equivalent, we can find $\delta_k>0$ such that
	\begin{equation}\label{PT.Super.equi-norms}
	\delta_k\|u\|\leq \|u\|_{L^{p^+_B}(m,B)},\ \ \forall u\in E_k.
	\end{equation}
	Let $\gamma_k>0$ be  positive constant such that
	\begin{equation}\label{PT.super1.xim}
	\lambda\gamma_k \delta_k^{p^+_B}-\widetilde{C}>0
	\end{equation}
	with $\widetilde{C}$ in \eqref{A2'}. Then, by $\textup{(F5)}$ we find $T_k>0$ such that
	\begin{equation*}
	F(x,\tau)\geq \gamma_k m(x)|\tau|^{p^+_B},\ \ \text{for a.e.}\, x\in B\  \text{and all}\ |\tau|\geq T_k\, .
	\end{equation*}
	Hence,
	\begin{equation}\label{PT.super1.F(x,u)}
	F(x,\tau)\geq \gamma_k m(x)|\tau|^{p^+_B}-\sup_{|\tau|\leq T_k}\, |F(x,\tau)|,\ \ \text{for a.e.}\ x\in B\ \text{and all}\ \ \tau\in\mathbb{R}.
	\end{equation}
	From \eqref{A2'} and \eqref{PT.Super.equi-norms}-\eqref{PT.super1.F(x,u)}, for $u\in E_k\setminus B_{R_k}$ with $R_k>1$ large enough we have
	\begin{align*}
	I(u)&\leq \widetilde{C}\int_B \left(|\nabla u|^{p(x)}+1\right)\diff x-\lambda \int_{B}F(x,u)\diff x\notag\\
	&\leq \widetilde{C}\left(\|u\|^{p^+_B}+|B|\right)-\lambda\gamma_k\|u\|^{p^+_B}_{L^{p^+_B}(m,B)}+\lambda\int_{B} \sup_{|\tau|\leq T_k}\, |F(x,\tau)|\diff x\notag\\
	&\leq -\left(\lambda\gamma_k \delta_k^{p^+_B}-\widetilde{C}\right)\|u\|^{p^+_B}+\lambda\int_{B} \sup_{|\tau|\leq T_k}\, |F(x,\tau)|\diff x+\widetilde{C}|B|\\
	&\leq -\left(\lambda\gamma_k \delta_k^{p^+_B}-\widetilde{C}\right)R_k^{p^+_B}+\lambda\int_{B} \sup_{|\tau|\leq T_k}\, |F(x,\tau)|\diff x+\widetilde{C}|B|\notag\\
	&<0,
	\end{align*}
proving \eqref{PT.Super1.Rk}. Here we have used that fact that $\sup_{|\tau|\leq T_k}\, |F(\cdot,\tau)|\in L^1(B)$ due to $\textup{(F4)}$. Obviously, $R_k$ above can be chosen independently of $\theta$.
	\end{proof}
Finally, to verify the remaining conditions of Proposition~\ref{prop.abs}, we introduce a sequence  $\{V_k\}_{k=1}^\infty$ of finite dimensional linear subspaces of $E$ as follows.  Let $\{e_n\}_{n=1}^\infty\subset E$ be a Schauder basis of $E$ and let $\{f_n\}_{n=1}^\infty\subset E^*$ be such that for each $n\in\mathbb{N},$
$$\langle f_n,u\rangle=\alpha_n\quad \text{for}\ \ u=\sum_{j=1}^\infty\alpha_je_j\in E.$$
For each $k\in\mathbb{N},$ define
$$V_k:=\{u\in E:\ \langle f_n,u\rangle=0,\ \ \forall\, n> k\}$$
and 
$$W_k:=\{u\in E:\ \langle f_n,u\rangle=0,\ \ \forall\, n\leq k\}.$$
Then, $E=V_k \oplus W_k$ and $\operatorname{dim}V_k=k$. Define
\begin{equation}\label{PT.super1.delta_k}
\delta_k:=\underset{\|v\|\leq 1}{\underset{v\in W_k}{\sup}}\, \max_{1\leq j\leq m_0}\, \|v\|_{L^{r_j(\cdot)}(a_j,\Omega)}.
\end{equation}
The next lemma was proved in \cite{HHS}.

\begin{lemma}\label{PL.Super.delta_k}
	The sequence $\{\delta_k\}_{k=1}^\infty$ above satisfies $0< \delta_{k+1}\leq\delta_k$ for all $k\in\mathbb{N}$ and
	\begin{equation*}
	 \lim_{k\to\infty}\delta_k=0.
	\end{equation*} 
\end{lemma}
We are now in a position to complete  the proof of Theorem~\ref{Theo.Superlinear}.
 
\begin{proof}[\textbf{Proof of Theorem~\ref{Theo.Superlinear}}]  Let $\lambda\in (0,\mu_1]$ be given and let $\theta> 0$. We will show that  the conditions $(\textup{I}1)-(\textup{I}3)$ of Proposition~\ref{prop.abs} are fulfilled with $X=E$ and $I$ given by \eqref{def.I}. Firstly, we will show that $I$ satisfies $(\textup{I}1)$. By Lemma~\ref{PL.Super.delta_k}, we find $k_0\in\N$ such that $\delta_{k_0} \in (0,1)$. 
By using $(\textup{A3})$ and $\textup{(F4)}$ we have that for any $u\in W_{k_0}$ with $\|u\|>1,$ 
	\begin{align}\label{PT.Super.est1I}
	I(u)&\geq\frac{1}{p^+}\left(\|u\|^{p^-}-1\right)-\lambda\sum_{j=1}^{m_0}\frac{1}{r_j^-}\left(\|u\|_{L^{r_j(\cdot)}(a_j,\Omega)}^{r_j^+}+1\right) -\frac{\theta}{t^-}\int_\Omega b|u|^{t(x)}\diff x.	\end{align}
 Note that for $u\in E$ with $\|u\|>1$, it follows from Proposition~\ref{norm-modular} and \eqref{Sb} that
 \begin{equation}\label{PT.Super.est1I'}
 	\int_\Omega b|u|^{t(x)}\diff x\leq \left(1+S_b^{-t^+}\right)\|u\|^{t^+}.
 \end{equation} 
Moreover, by considering $v=z/\|z\|$ for $z\in W_{k_0}\setminus\{0\}$ we deduce from the definition \eqref{PT.super1.delta_k} of $\delta_k$ that
\begin{equation}\label{PT.Super.est1I''}
	\max_{1\leq j\leq m_0}\, \|z\|_{L^{r_j(\cdot)}(a_j,\Omega)}\leq \delta_{k_0} \|z\|,\quad \forall z\in W_{k_0}.
\end{equation}
Invoking \eqref{PT.Super.est1I'} and \eqref {PT.Super.est1I''}, we get from \eqref{PT.Super.est1I} that 
for all $u\in W_{k_0}$ with $\|u\|>1$,
\begin{align}\label{PT.Super.est2I}
		I(u)\geq\frac{1}{p^+}\|u\|^{p^-}-\frac{m_0\lambda\delta_{k_0}^{r_\ast}}{r_\ast}\|u\|^{r^\ast}-\left(\frac{m_0\lambda}{r_\ast}+\frac{1}{p^+}\right)-\frac{ \left(1+S_b^{-t^+}\right)\theta}{t^-}\|u\|^{t^+},
\end{align}
where $r_\ast:=\min_{1\leq j\leq m_0}\, r_j^-$ and $r^* := \max_{1 \le j \le m_0} r_j^+$. Let $\rho_{k_0}>0$ be such that
$$\frac{m_0\lambda\delta_{k_0}^{r_\ast}}{r_\ast}\rho_{k_0}^{r^\ast}=\frac{1}{2p^+}\rho_{k_0}^{p^-}\ \ \text{i.e.,}\ \ \rho_{k_0}=\left(\frac{r_*}{2m_0p^+\lambda\delta_{k_0}^{r^*}}\right)^{\frac{1}{r^*-p^-}}.$$
By Lemma~\ref{PL.Super.delta_k}, we can fix $k_0\in\mathbb{N}$ from the beginning such that
\begin{equation*}
	0<\delta_{k_0}<1<\rho_{k_0}\ \ \text{and}\ \ \frac{1}{2p^+}\rho_{k_0}^{p^-}-\left(\frac{m_0\lambda}{r_\ast}+\frac{1}{p^+}\right)>\frac{1}{4p^+}\rho_{k_0}^{p^-}.
\end{equation*}
Thus, \eqref{PT.Super.est2I} yields
\begin{equation*}
	I(u)\geq \frac{1}{4p^+}\rho_{k_0}^{p^-}-\frac{ \left(1+S_b^{-t^+}\right)\theta}{t^-}\rho_{k_0}^{t^+},\quad \forall u\in \partial B_{\rho_{k_0}}\cap W_{k_0}.
\end{equation*}
Therefore, by choosing $V:=V_{k_0}$, $W:=W_{k_0}$ and $\theta_0:=\frac{t^-\rho_{k_0}^{p^--t^+}}{4p^+\left(1+S_b^{-t^+}\right)}>0$, we have that for any $\theta\in (0,\theta_0)$,
$$I(u)\geq \beta,\ \ \forall u\in \partial B_\rho\cap W$$
with $\rho=\rho_{k_0}$ and $\beta=\frac{ \left(1+S_b^{-t^+}\right)\rho_{k_0}^{t^+}}{t^-}(\theta_0-\theta)>0.$ That is, $I$ satisfies $(\textup{I}1)$ in Proposition~\ref{prop.abs}.

\vspace{0.3cm}
Next, we will show that $I$ satisfies $(\textup{I}2)$ in Proposition~\ref{prop.abs}. To this end, let $\{E_k\}_{k=1}^\infty$ be defined by \eqref{Ek}. For each $k\in\N$ we have 
\begin{align}\label{PT.Super.Lk*}
	\sup_{u\in E_{k}}\, I(u)=\underset{\|u\|\leq R_k}{\max_{u\in E_k}}\, I(u)
\end{align}
due to the fact that $I(0)=0$ and Lemma~\ref{PL.Super.Rk}. Note that by using \eqref{A2'} and $\textup{(F4)}$, we have that for any $u\in E$ with $\|u\|\leq R_k$,
\begin{align*}
	I(u)&< \widetilde{C}\int_\Omega \left(|\nabla u|^{p(x)}+1\right)\diff x-\lambda \int_{\Omega}F(x,u)\diff x\notag\\
	&\leq \widetilde{C}\|u\|^{p^+}+\widetilde{C}(|\Omega|+1)+\lambda \sum_{j=1}^{m_0}\int_{\Omega}\frac{a_j(x)}{r_j(x)}|u|^{r_j(x)}\diff x\notag\\
	&\leq \widetilde{C}\|u\|^{p^+}+\widetilde{C}(|\Omega|+1)+ \sum_{j=1}^{m_0}\frac{\lambda}{r_j^-}\left(\|u\|_{L^{r_j(\cdot)}(a,\Omega)}^{r_j^+}+1\right)\notag\\
	&\leq \widetilde{C}\|u\|^{p^+}+\frac{\lambda}{r_\ast} \sum_{j=1}^{m_0} C_j^{r_j^+}\|u\|^{r_j^+}+\frac{m_0\lambda}{r_\ast}+\widetilde{C}(|\Omega|+1)\notag\\
	&\leq \widetilde{C}R_k^{p^+}+\frac{\lambda}{r_\ast} \sum_{j=1}^{m_0} C_j^{r_j^+}R_k^{r_j^+}+m_0\lambda+\widetilde{C}(|\Omega|+1)=:L_k,
\end{align*}
where  $C_j$ ($j=1,\cdots,m_0$) denotes the imbedding constant for $E\hookrightarrow L^{r_j(\cdot)}(a_j,\Omega)$.
From this and \eqref{PT.Super.Lk*} we obtain
\begin{align}\label{PT.Super.Lk**}
	\sup_{u\in E_{k}}\, I(u)<L_k, \quad \forall k\in\N.
\end{align}
Finally, let $\{\theta_k\}_{k=1}^\infty\subset (0,\theta_0)$ be such that for any $k\in \mathbb{N},$
\begin{equation}
	\begin{cases}
		L_{k_0+k}<-\frac{\|e\|_{L^1(\Omega)}\lambda}{\alpha}+\left(\frac{1}{\alpha}-\frac{1}{t^-}\right)S_b^N\min\left\{\theta_k^{-\frac{1}{h^+-1}},\theta_k^{-\frac{1}{h^--1}}\right\},\\
		\theta_{k+1}<\theta_k.
	\end{cases}
\end{equation}
Then, for any $\theta\in (\theta_{k+1},\theta_k)$ we have
$$L_{k_0+k}<-\frac{\|e\|_{L^1(\Omega)}\lambda}{\alpha}+\left(\frac{1}{\alpha}-\frac{1}{t^-}\right)S_b^N\min\left\{\theta^{-\frac{1}{h^+-1}},\theta^{-\frac{1}{h^--1}}\right\}$$
and hence, $I$ satisfies the $\textup{(PS)}_c$ for any $c\in [0,L_{k_0+k}]$ in view of Lemma~\ref{le.ps1}. Thus, $I$ satisfies $(\textup{I}2)$ and $(\textup{I}3)$ with $\widetilde{X}=E_{k_0+k}$ and $L=L_{k_0+k}.$ So, $I$ admits at least $\operatorname{dim}\, \widetilde{X}- \operatorname{dim}\, V=k$  pairs of nontrivial critical points in view of Proposition~\ref{prop.abs}; hence, problem~\eqref{Eq.General} has at least $k$ pairs of nontrivial solutions. The proof is complete.
	\end{proof}
\begin{remark}\rm
	We would like to discuss problem~\eqref{Eq.General} without a parameter in the subcritical term as follows:
		\begin{align}\label{Eq.General*}
		\begin{cases}
			-\operatorname{div}a(x,\nabla u)=f(x,u)+\theta b(x)|u|^{t(x)-2}u \quad \text{in } \Omega,\\
			u=0\quad \text{on } \partial \Omega. 
		\end{cases}
	\end{align}
The existence of a sequence of nontrivial solutions to problem~\eqref{Eq.General*} for the $p(\cdot)$-sublinear case was obtained in Theorem~\ref{Theo.Sublinear} (for all $\theta\geq 0$). For the $p(\cdot)$-superinear case, by repeating the proof of Theorem~\ref{Theo.Superlinear}, we find a sequence $\{\theta_n\}_{n=1}^\infty$ with $0<\theta_{n+1}<\theta_n$ for all $n\in\mathbb{N}$ such that for any $\theta\in (\theta_{n+1},\theta_n)$, problem \eqref{Eq.General*} admits at least $n$ distinct pairs of nontrivial solutions under  $\textup{(P)}, (\textup{A0})-(\textup{A4})$, $\textup{(F3)}-\textup{(F5)}$ and $\textup{(F6)}$ replaced by 
\begin{itemize}
		\item [$(\widetilde{\textup{F}}6)$] For $w$ given by $\textup{(P)}$, there exist $\alpha\in [p^+,t^-)$ and $e\in L^1_+(\Omega)$ such that $$\alpha F(x,\tau)-\tau f(x,\tau)\leq\frac{\left(\alpha-p^+\right)\mu_1}{p^+}w(x)|\tau|^{p(x)}+e(x)\  \text{for a.e.}\ x\in\Omega\ \text{and all}\ \tau\in\mathbb{R}.$$
\end{itemize}
The result remains valid under $(\textup{A0})-(\textup{A4})$, $\textup{(F3)}-\textup{(F5)}$ and $\textup{(F6)}$ replaced by 
\begin{itemize}
	\item [$(\overline{\textup{F}}6)$] There exist $\alpha\in [p^+,t^-)$ and functions $\sigma\in C_+(\overline{\Omega})$ with $\sigma^+<p^-$, $w\in L_+^{\frac{t(\cdot)}{t(\cdot)-\sigma(\cdot)}}\left(b^{-\frac{\sigma}{t-\sigma}},\Omega\right)$, and $e\in L^1_+(\Omega)$ such that 
	$$\alpha F(x,\tau)-\tau f(x,\tau)\leq w(x)|\tau|^{\sigma(x)}+e(x)\  \text{for a.e.}\ x\in\Omega\ \text{and all}\ \tau\in\mathbb{R}.$$
\end{itemize} 
This result generalizes of \cite[Theorem A]{Sil-Xav.2003} that treated for the $p$-Laplacian.
\end{remark}


\section{The sandwich case}\label{sandwich}
 Throughout this section we always assume that $(\textup{A0})-(\textup{A5})$, $\textup{(S)}$, $\textup{(W)}$ and $\textup{(L1)}$ hold. As in Subsection~\ref{Subsec3.2}, we will restrict our proof to the case $p^+<N$. To determine (nonnegative) solutions of problem~\eqref{Eq.Sandwich}, we introduce the functionals $J,\Psi:\, E\to\mathbb{R}$ defined as
\begin{equation}\label{J}
	J(u) =\int_\Omega A(x,\nabla u)\diff x-\lambda\int_\Omega \frac{ m(x)}{s(x)}u_+^{s(x)}\diff x-\theta\int_\Omega \frac{b(x)}{t(x)}u_+^{t(x)}\diff x
\end{equation}
and
\begin{equation}\label{Psi}
	\Psi(u) =\int_\Omega A(x,\nabla u)\diff x-\lambda\int_\Omega \frac{ m(x)}{s(x)}|u|^{s(x)}\diff x-\theta\int_\Omega \frac{b(x)}{t(x)}|u|^{t(x)}\diff x
\end{equation}
for $ u\in E$.
Here and in the rest of paper, both $\phi_+$ and $\phi^+$ stand for $\max\{\phi,0\}$. Under the given assumptions, $J,\Phi$ are of class $C^1\left(E,\mathbb{R}\right)$ and a critical point of $J$ (resp. $\Phi$) is a nonnegative solution (resp. a solution) to problem~\eqref{Eq.Sandwich}. We will make use of the critical points theorems and genus theory to determine critical points of $J$ and $\Phi$. In order to verify the $\textup{(PS)}$ condition for $J$ and $\Psi$, we need the next lemma.

\begin{lemma}\label{le.sandwich.PS} The following assertions hold.
	\begin{itemize}
		\item [(i)] For $\theta=0$ and $\lambda\in\mathbb{R},$ $J$ and $\Psi$ satisfy the $\textup{(PS)}_c$ condition for all $c\in\mathbb{R}.$
		\item  [(ii)] For $\theta>0$ and $\lambda\in\mathbb{R},$ $J$ and $\Psi$ satisfy the $\textup{(PS)}_c$ condition for all $c\in\mathbb{R}$ satisfying
		\begin{align}\label{PS_c2}
		c<\left(\frac{1}{p^+}-\frac{1}{t^-}\right)S_b^N&\min\left\{\theta^{-\frac{1}{h^+-1}},\theta^{-\frac{1}{h^--1}}\right\}\notag\\
		&-K\max\left\{\theta^{-\frac{1}{\ell^+-1}},\theta^{-\frac{1}{\ell^--1}}\right\}\max\left\{|\lambda|^{\frac{\ell^+}{\ell^+-1}},|\lambda|^{\frac{\ell^-}{\ell^--1}}\right\},
		\end{align}
	where $K=K(p,s,t,m,b,\Omega)>0$, $h(x):=\frac{t(x)}{p(x)}$ and $\ell(x):=\frac{t(x)}{s(x)}$ for $x\in\overline{\Omega}.$
			\end{itemize}
		
	\end{lemma}
\begin{proof}
	We only prove the conclusions for $J$ since those of $\Psi$ are proved similarly. We first show that for $\lambda\in\mathbb{R},\,  \theta\geq 0, $ and $c\in\mathbb{R},$ any $\textup{(PS)}_c$-sequence of $J$ is bounded in $E.$ Indeed, let $\{u_n\}_{n=1}^\infty$ be a $\textup{(PS)}_c$-sequence of $J$.  Then, for $n$ large we have
	\begin{eqnarray*}\label{PL.sand.Est}
		\notag c+1+\|u_n\|&\geq& J(u_n)-\frac{1}{t^-}\langle J'(u_n) ,u_n \rangle \\
		&\geq& \left (\frac{1}{p^+}-\frac{1}{t^-}\right)\int_{\Omega} |\nabla u_n |^{p(x)}\diff x-|\lambda|\left (\frac{1}{s^-}-\frac{1}{t^-}\right)\int_{\Omega}|m(x)||u_n|^{s(x)}\diff x.
	\end{eqnarray*}
On the other hand, by invoking Propositions~\ref{Holder ineq}-\ref{norm-norm} and utilizing \eqref{Sb} we have
\begin{equation*}
\int_{\Omega} |\nabla u_n |^{p(x)}\diff x\geq \|u_n\|^{p^-}-1
\end{equation*}
and
	\begin{align*}\label{Est.int}
		\int_\Omega|m(x)||u_n|^{s(x)}\diff x&\leq 2\left\|mb^{-1}\right\|_{L^{\frac{t(\cdot)}{t(\cdot)-s(\cdot)}}(b,\Omega)}\left\||u_n|^s\right\|_{L^{\frac{t(\cdot)}{s(\cdot)}}(b,\Omega)}\\
		&\leq 2\left\|mb^{-1}\right\|_{L^{\frac{t(\cdot)}{t(\cdot)-s(\cdot)}}(b,\Omega)}\left(\|u_n\|_{L^{t(\cdot)}(b,\Omega)}^{s^+}+1\right)\\
		&\leq 2 \left\|mb^{-1}\right\|_{L^{\frac{t(\cdot)}{t(\cdot)-s(\cdot)}}(b,\Omega)}\left(S_b^{-s^+}\|u_n\|^{s^+}+1\right).
	\end{align*}
From the last three estimations we deduce the boundedness of $\{u_n\}_{n=1}^\infty$ in $E$ since $1<s^+<p^-$. Hence, up to a subsequence we have
\begin{equation}\label{PL.sandwich.weak1}
u_n\rightharpoonup u\ \ \text{in}\ E\ \text{as}\ n\to\infty.	
\end{equation}
(i) \textit{Case $\theta=0,\, \lambda\in\mathbb{R}.$}
Since 
	\begin{align*}
	\int_{\Omega}a(x,\nabla u_n)\cdot (\nabla u_n-\nabla u)\diff x=\langle J'(u_n) ,u_n-u \rangle+\lambda\int_{\Omega}m(x)(u_n^+)^{s(x)-1}(u_n-u)\diff x, 
	\end{align*}
we easily obtain from \eqref{PL.sandwich.weak1} and Proposition~\ref{le.compact.imb} that 
$$\int_{\Omega}a(x,\nabla u_n)\cdot (\nabla u_n-\nabla u)\diff x\to 0\ \text{as}\ n\to\infty.$$
Then, as in the proof of Lemma~\ref{le.ps1} we obtain $u_n\to u$ in $E$ as $n\to\infty$.

\vspace{0.3cm}
(ii) \textit{Case $\theta>0,\, \lambda\in\mathbb{R}.$} By Theorem~\ref{T.CCP} again, we infer from the boundedness of $\{u_n\}_{n=1}^\infty$ that up to a subsequence, 
\begin{eqnarray}
u_n(x) &\to& u(x)  \quad \text{a.e.} \ \ x\in\Omega,\label{PL.sandwich.a.e1}\\
|\nabla u_n|^{p(x)} &\overset{\ast }{\rightharpoonup }&\mu \geq |\nabla u|^{p(x)} + \sum_{i\in \mathcal{I}} \mu_i \delta_{x_i} \ \text{in}\  \mathcal{M}(\overline{\Omega}),\label{PL.sandwich.mu1}\\
b|u_n|^{t(x)}&\overset{\ast }{\rightharpoonup }&\nu=b|u|^{t(x)} + \sum_{i\in \mathcal{I}}\nu_i\delta_{x_i} \ \text{in}\ \mathcal{M}(\overline{\Omega}),\label{PL.sandwich.nu1}\\
S_b \nu_i^{1/t(x_i)} &\leq& \mu_i^{1/p(x_i)}, \quad \forall i\in \mathcal{I}.
\end{eqnarray}
We claim that $\mathcal{I}=\emptyset.$ Suppose by contradiction that there exists $i\in\mathcal{I}$. Arguing as that obtained \eqref{PT.super.ps.mu-nu}, we obtain
\begin{equation}\label{PL.Sandwich.mu-nu}
\theta \nu_i\geq S_b^N\min\left\{\theta^{-\frac{1}{h^+-1}},\theta^{-\frac{1}{h^--1}}\right\}.
\end{equation}
On the other hand, we have
\begin{align*}
c&=J(u_n)-\frac{1}{p^+}\langle J'(u_n) ,u_n\rangle+o_n(1)\\
&\geq -|\lambda| \left(\frac{1}{s^-}-\frac{1}{p^+}\right)\int_{\Omega}|m(x)||u_n|^{s(x)}\diff x+\theta\left(\frac{1}{p^+}-\frac{1}{t^-}\right)\int_{\Omega}b(x)|u_n|^{t(x)}\diff x+o_n(1).
\end{align*}
Passing to the limit as $n\to\infty$ in the last estimation and invoking Proposition~\ref{le.compact.imb} again and using \eqref{PL.sandwich.nu1}, we obtain
\begin{equation*}\label{PL.PS2.c}
c\geq -|\lambda| \left(\frac{1}{s^-}-\frac{1}{p^+}\right)\int_{\Omega}|m(x)||u|^{s(x)}\diff x+\theta\left(\frac{1}{p^+}-\frac{1}{t^-}\right)\left(\int_{\Omega}b(x)|u|^{t(x)}\diff x+\nu_i\right).
\end{equation*}
Invoking \eqref{PL.Sandwich.mu-nu} the preceding inequality gives
$$c\geq \left(\frac{1}{p^+}-\frac{1}{t^-}\right)\theta\int_{\Omega}b(x)|u|^{t(x)}\diff x-|\lambda| \left(\frac{1}{s^-}-\frac{1}{p^+}\right)\int_{\Omega}|m(x)||u|^{s(x)}\diff x+k(\theta),$$
where $k(\theta):=\left(\frac{1}{p^+}-\frac{1}{t^-}\right)S_b^N\min\left\{\theta^{-\frac{1}{h^+-1}},\theta^{-\frac{1}{h^--1}}\right\}$. From Proposition~\ref{Holder ineq}, we have
\begin{equation}\label{S4.int-mu^s}
	\int_{\Omega}|m||u|^{s(x)}\diff x=\int_{\Omega}b\left|mb^{-1}\right||u|^{s(x)}\diff x\leq 2\|mb^{-1}\|_{L^{\frac{\ell(\cdot)}{\ell(\cdot)-1}}(b,\Omega)}\big\||u|^{s}\big\|_{L^{\ell(\cdot)}(b,\Omega)},
\end{equation}
where $\ell (x):=\frac{t(x)}{s(x)}.$ From the last two inequalities, we obtain
\begin{equation}\label{5.estc}
c\geq a_1\theta\int_{\Omega}b\big(|u|^{s(x)}\big)^{\ell(x)}\diff x-b_1|\lambda| \big\||u|^{s}\big\|_{L^{\ell(\cdot)}(b,\Omega)}+ k(\theta),
\end{equation}
where $a_1:=\frac{1}{p^+}-\frac{1}{t^-}>0$ and $b_1:=2\left(\frac{1}{s^-}-\frac{1}{p^+}\right)\|mb^{-1}\|_{L^{\frac{\ell(\cdot)}{\ell(\cdot)-1}}(b,\Omega)}>0.$
\begin{itemize}
	\item If $\big\||u|^{s}\big\|_{L^{\ell(\cdot)}(b,\Omega)}\geq 1,$ then \eqref{5.estc} and Proposition~\ref{norm-modular} yield
\end{itemize}
$$c\geq a_1\theta \xi^{\ell^-}-b_1|\lambda| \xi+k(\theta)=:g_1(\xi)\ \text{with}\ \xi=\big\||u|^{s}\big\|_{L^{\ell(\cdot)}(b,\Omega)}\geq 1.$$
Thus,
$$c\geq \underset{\xi\geq 0}{\min}\ g_1(\xi)=g_1\left(\big(\frac{b_1|\lambda|}{\ell^-a_1\theta}\big)^{\frac{1}{\ell^--1}}\right),$$
i.e.,
$$c\geq k(\theta)-(\ell^--1)(\ell^-)^{-\frac{\ell^-}{\ell^--1}}a_1^{-\frac{1}{\ell^--1}}b_1^{\frac{\ell^-}{\ell^--1}}\theta^{-\frac{1}{\ell^--1}}|\lambda|^{\frac{\ell^-}{\ell^--1}}.$$
\begin{itemize}
	\item If $\big\||u|^{s}\big\|_{L^{\ell(\cdot)}(b,\Omega)}<1,$ then \eqref{5.estc} and Proposition~\ref{norm-modular} yield
\end{itemize}
$$c\geq a_1\theta \xi^{\ell^+}-b_1\lambda \xi+k(\theta)=:g_2(\xi)\ \text{with}\ \xi=\big\||u|^{s}\big\|_{L^{\ell(\cdot)}(b,\Omega)}\in [0,1).$$
Thus,
$$c\geq \underset{\xi\geq 0}{\min}\ g_2(\xi)=g_2\left(\big(\frac{b_1|\lambda|}{\ell^+a_1\theta}\big)^{\frac{1}{\ell^+-1}}\right),$$
i.e.,
$$c\geq k(\theta)-(\ell^+-1)(\ell^+)^{-\frac{\ell^+}{\ell^+-1}}a_1^{-\frac{1}{\ell^+-1}}b_1^{\frac{\ell^+}{\ell^+-1}}\theta^{-\frac{1}{\ell^+-1}}|\lambda|^{\frac{\ell^+}{\ell^+-1}}.$$
Therefore, in any case, we obtain
$$c\geq k(\theta)-K\max\left\{\theta^{-\frac{1}{\ell^+-1}},\theta^{-\frac{1}{\ell^--1}}\right\}\max\left\{|\lambda|^{\frac{\ell^+}{\ell^+-1}},|\lambda|^{\frac{\ell^-}{\ell^--1}}\right\},$$
where $K:=\underset{*\in\{+,-\}}{\max}(\ell^*-1)(\ell^*)^{-\frac{\ell^*}{\ell^*-1}}a_1^{-\frac{1}{\ell^*-1}}b_1^{\frac{\ell^*}{\ell^*-1}}>0$. This contradicts to \eqref{PS_c2} and hence, $\mathcal{I}=\emptyset$. It follows that \ $\int_{\Omega}b(x)|u_n|^{t(x)}\diff x\to \int_{\Omega}b(x)|u|^{t(x)}\diff x$\ as $n\to\infty.$ We then argue as in the proof of Lemma~\ref{le.ps1} to conclude that $u_n\to u$ in $E$ as $n\to\infty.$ The proof is complete.

	\end{proof}


\subsection{Multiplicity of solutions}
In this subsection we borrow the idea from \cite{BBF.2021} to prove Theorem~\ref{Theo.Sandwich-infty}, which asserts the existence of $k$ pairs of solutions  to problem~\eqref{Eq.Sandwich} for an arbitrarily given number $k$. To seek such a $k$ pairs of solutions, we determine the critical points of the functional $\Psi$ defined in \eqref{Psi}. Throughout this subsection, in addition to $\textup{(S)}$, $\textup{(W)}$ and $\textup{(L1)}$ we always assume that $m(x)>0$ a.e. in $B$. We will frequently make use of the following estimates that can be easily obtained by using \eqref{PT.Super.est1I'}, \eqref{S4.int-mu^s} and Proposition~\ref{norm-norm}: there exists a constant $C_1>1$ such that for all $u\in E$ with $\|u\|>1,$
\begin{equation}\label{S4.Est.ints}
	\int_{\Omega}|m(x)||u|^{s(x)}\diff x\leq C_1^{s^+}\|u\|^{s^+}\ \text{and} \  \int_\Omega b(x)|u|^{t(x)} \diff x \leq C_1^{t^+}\|u\|^{t^+}.
\end{equation}
Let $\lambda,\theta>0$ to be specified later. By $(\textup{A3})$, we have
\begin{align*}
	\notag \Psi(u)&\geq \int_\Omega \frac{1}{p(x)}|\nabla u|^{p(x)} \diff x-\lambda\int_\Omega \frac{|m(x)|}{s(x)}|u|^{s(x)}\diff x-\theta\int_\Omega \frac{b(x)}{t(x)}|u|^{t(x)} \diff x\\
	&\geq \frac{1}{p^+}\int_\Omega |\nabla u|^{p(x)} \diff x-\frac{\lambda}{s^-} \int_\Omega |m(x)||u|^{s(x)}\diff x-\frac{\theta}{t^-}\int_\Omega b(x)|u|^{t(x)} \diff x,\quad \forall u\in E.
\end{align*}
By invoking Proposition~\ref{norm-modular} then using \eqref{S4.Est.ints} we deduce from the last inequality that
\begin{align}
	\notag \Psi(u)\geq\frac{1}{p^+}\|u\|^{p^-}-\frac{ C_1^{s^+}\lambda}{s^-}\|u\|^{s^+}-\frac{C_1^{t^+}\theta}{t^-}\|u\|^{t^+}\ \ \text{for all} \ u\in E\ \text{with}\ \|u\|\geq 1.
\end{align}
Thus,
\begin{equation}\label{PTcc.gi}
	\Psi(u)\geq  g_\lambda(\|u\|)\ \ \text{for all} \ u\in E\ \text{with}\ \|u\|\geq1,
\end{equation} 
where $g_\lambda\in C[0,\infty)$ is given by
$$g_\lambda(\tau):=\frac{1}{p^+}\tau^{p^-}-\frac{ C_1^{s^+}\lambda}{s^-}\tau^{s^+}-\frac{C_1^{t^+}\theta}{t^-}\tau^{t^+}.$$
Rewrite $g_\lambda(\tau)=\frac{C_1^{s^+}\tau^{s^+}}{s^-}\left(h(\tau)-\lambda\right)$ with $$h(\tau):=a_0\tau^{p^--s^+}-b_0\theta \tau^{t^+-s^+},$$ 
where 
\begin{equation}\label{a0}
a_0:=(p^+C_1^{s^+})^{-1}s^->0,\ b_0:=(t^-)^{-1}s^-C_1^{t^+-s^+}>0.
\end{equation}
It is clear that 
\begin{align}\label{Sandwich2.lambda3}
	\notag\max_{\tau\geq 0}\, h(\tau)=h\left(\left[\frac{(p^--s^+)a_0}{(t^+-s^+)b_0\theta}\right]^{\frac{1}{t^+-p^-}}\right)=c_0\theta^{\frac{s^+-p^-}{t^+-p^-}}>0,
\end{align}
where $c_0:=a_0^{\frac{t^+-s^+}{t^+-p^-}}b_0^{\frac{s^+-p^-}{t^+-p^-}}\left(\frac{p^--s^+}{t^+-s^+}\right)^{\frac{p^--s^+}{t^+-p^-}}\frac{t^+-p^-}{t^+-s^+}>0.$ Thus, for $\lambda\in (0,\lambda_0(\theta))$ with
\begin{equation}\label{S4.lambda0(theta)}
\lambda_0(\theta):=c_0\theta^{\frac{s^+-p^-}{t^+-p^-}},
\end{equation}
$g_\lambda(\tau)$ has only positive roots $T_0$ and $T_1$ with
\begin{equation}\label{T0-1}
	0<T_0<T_*:=\left[\frac{(p^--s^+)a_0}{(t^+-s^+)b_0\theta}\right]^{\frac{1}{t^+-p^-}}=d_0\theta^{-\frac{1}{t^+-p^-}}<T_1,
\end{equation}
where $d_0:=\left[\frac{(p^--s^+)a_0}{(t^+-s^+)b_0}\right]^{\frac{1}{t^+-p^-}}>0$. Obviously, on $[0,\infty)$ the function $g_\lambda(\tau)$ is negative only on $(0,T_0)\cup (T_1,\infty)$. It is clear that \begin{equation}\label{PT.Sand.T0}
	T_0>\left(\frac{\lambda}{a_0}\right)^{\frac{1}{p^--s^+}}
\end{equation}
 and for $\theta\in (0,\theta_*)$ with
\begin{equation}\label{PT.Sand.lbd1}
\theta_*:=c_0^{\frac{t^+-p^-}{p^--s^+}}a_0^{\frac{t^+-p^-}{s^+-p^-}},
\end{equation}
we have that $\lambda_0(\theta)>a_0$. From the above analysis, we infer that for $\theta\in (0,\theta_*)$ and $\lambda\in (a_0, \lambda_0(\theta))$,  $g_\lambda(\tau)$ has only positive roots $T_0$ and $T_1$ with $1<T_0<T_*<T_1$.

 \vspace{0.3cm}
 Let $\theta\in (0,\theta_*)$. For each $\lambda\in (a_0, \lambda_0(\theta))$, we define the truncated functional $T_\lambda: E\to\R$ as
\begin{align*}\label{T_lambda}
	\notag T_\lambda(u):=&\int_\Omega A(x,\nabla u) \diff x-\lambda\int_\Omega \frac{m(x)}{s(x)}|u|^{s(x)} \diff x-\phi(\|u\|) \theta \int_\Omega \frac{b(x)}{t(x)}|u|^{t(x)} \diff x,\ u\in E,
\end{align*}
where $\phi\in C_c^\infty(\R)$, $0\leq \phi(\tau)\leq 1$ for all $\tau\in\R$, $\phi(\tau)=1$ for $\tau\leq T_0$ and $\phi(\tau)=0$ for $\tau\geq T_1$. Note that $u\mapsto \phi (\|u\|)$ is of class $C^1(E,\R)$ due to \cite[Corollary 3.3]{Lang2016} and therefore, $T_\lambda\in C^1(E,\R)$. Moreover, it holds that
\begin{equation}\label{T_lambda.Est1}
	T_\lambda(u)\geq \Psi(u)\ \ \text{for all} \ \ u\in E,
\end{equation}
\begin{equation}\label{T_lambda.Est2}
	T_\lambda(u)=\Psi(u) \ \ \text{for all} \ \ u\in E \ \ \text{with} \ \ \|u\|<T_0,
\end{equation}
and
\begin{equation}\label{T_lambda.Est3}
	T_\lambda(u)=\int_\Omega A(x,\nabla u) \diff x-\lambda\int_\Omega \frac{m(x)}{s(x)}|u|^{s(x)} \diff x \ \ \text{for all} \ \ u\in E \ \ \text{with} \ \ \|u\|>T_1.
\end{equation}

We have the following.
\begin{lemma}\label{T_lambda(u)<0}
	Let $\theta\in (0,\theta_*)$ and let $\lambda\in (a_0,\lambda_0(\theta))$. Then, $T_\lambda(u)<0$ implies that $\|u\|<T_0$ and hence, $T_\lambda(u)=\Psi(u)$.
\end{lemma}
\begin{proof}
	Suppose that $T_\lambda(u)<0$. Then, we have $\Psi(u)<0$ due to \eqref{T_lambda.Est1}. Suppose on the contrary that  $\|u\|\geq T_0$. Note that $T_0>1$ and hence, \eqref{PTcc.gi} gives $g_\lambda(\|u\|)<0$. Thus, $\|u\|>T_1$ and hence; $T_\lambda(u)=\int_\Omega A(x,\nabla u) \diff x-\lambda\int_\Omega \frac{m(x)}{s(x)}|u|^{s(x)} \diff x<0$ due to \eqref{T_lambda.Est3}. On the other hand, the last inequality combining with $(\textup{A3})$ and \eqref{S4.Est.ints} gives
	\begin{equation*}
		\frac{1}{p^+}\|u\|^{p^-}\leq \frac{1}{s^-}C_1^{s^+}\lambda\|u\|^{s^+}.
	\end{equation*}
This yields
\begin{equation*}
T_0\leq \|u\|\leq \left(\frac{p^+C_1^{s^+}\lambda}{s^-}\right)^{\frac{1}{p^--s^+}},
\end{equation*}
which contradicts to \eqref{PT.Sand.T0}. Thus, $\|u\|<T_0$ and hence, $T_\lambda(u)=\Psi(u)$ due to \eqref{T_lambda.Est2}.
The proof is complete.
	\end{proof}
	
Next, by employing genus theory we will show that for any given number $k\in\N$, $T_\lambda$ admits at least $k$ pairs of critical points $\{\pm u_n\}_{n=1}^k$ with $T_\lambda(\pm u_n)<0$ provided $\theta$ is sufficiently small. Let $\Sigma$ denote the family of sets $A\subset E\setminus\{0\}$ such that $A$ is closed in $E$ and symmetric with respect to the origin. For $A\in\Sigma$, define the genus of $A$ to be $n$ (denoted by $\gamma(A)=n$) if there is a map $\varphi\in C(A,\R^n\setminus\{0\})$ and $n$ is the smallest integer with this property. When there does not exist a finite such n, set $\gamma(A)=\infty$. Finally set $\gamma(\emptyset) = 0.$  For each $k\in\N$, define
$$\Sigma_k:=\{A\in\Sigma:\ \gamma(A)\geq k\}.$$ 
Clearly, $\partial(H\cap B_\rho)\in\Sigma_k$ for any $k$-dimentional subspace $H$ of $E$ and for any $\rho>0$ (see \cite[Proposition 7.7]{Rab1986}). Thus, we can define
\begin{equation}\label{S4.c_k}
	c_{k}(\lambda,\theta):= \inf_{A\in \Sigma_k}\, \sup_{u \in
		A}\, T_\lambda(u).
\end{equation}
Clearly, $c_{k}(\lambda,\theta)\leq 	c_{k+1}(\lambda,\theta)$ for all $k\in\N$. We will show that with $\lambda, \theta$ specified later, $c_{k}(\lambda,\theta)$ are negative critical values of $T_\lambda$ and hence, of $\Psi$. To this end, let $E_{k}$ be a $k$-dimensional subspace of $E$ given by \eqref{Ek} with $B$ in $(\textup{L}1)$. Since all norms on $E_k$ are mutually equivalent, we find $\delta_k>1$ such that
\begin{equation}\label{equi.norms}
	\delta_k^{-1}\max\left\{\|\nabla u\|_{L^{p_B^+}(B)},\|\nabla u\|_{L^{q_B}(B)}\right\}\leq \|u\|\leq \delta_k\|u\|_{L^{s(\cdot)}(m,B)},\quad \forall u\in E_k.
\end{equation}
It is clear that we can choose $\{\delta_k\}_{k=1}^\infty$ such that $\delta_k<\delta_{k+1}$ for all $k\in\N$.  Recalling $a_0$ given by \eqref{a0}, we have the following.
 \begin{lemma}\label{Le.c_k<0}
	There exists $\{\lambda_k\}_{k=1}^\infty$ with $\max\{1,a_0\}<\lambda_k<\lambda_{k+1}$ for all $k\in\N$ such that for each $k\in\N$,  it holds that
	\begin{equation*}
		-\infty<c_k(\lambda,\theta)<0
	\end{equation*} 
for any $\lambda>\lambda_k$ and $\theta>0$.
 \end{lemma}
\begin{proof}
	Let $k\in\N$ be arbitrary and fixed and let $\lambda,\theta>0$ with $\lambda$ be specified later. By the definition of $T_{\lambda}$ and the fact that $s^+<p^-$, we easily see that $T_{\lambda}$ is bounded from below and hence, $$c_k(\lambda,\theta)>-\infty.$$ 
Let $u\in E_k$ with $\|u\|=\tau>\delta_k$ (hence, $\|u\|_{L^{s(\cdot)}(m,B)}>1$ due to \eqref{equi.norms}). We have
	\begin{align*}
		T_\lambda(u)\leq C_p\|\nabla u\|_{L^{p_B^+}(B)}^{p^+_B}+C_B\|\nabla u\|_{L^{q_B}(B)}^{q_B}-\frac{\lambda }{s^+}\|u\|_{L^{s(\cdot)}(m,B)}^{s_B^-}
	\end{align*}
in view of \eqref{A6} and Proposition~\ref{norm-modular}.
From this and \eqref{equi.norms} we obtain
	\begin{align}\label{T}
		T_\lambda(u)\leq C_p(\delta_k\tau)^{p^+_B}+C_B(\delta_k\tau)^{q_B}-\frac{\lambda }{s^+}(\delta_k^{-1}\tau)^{s_B^-}=\tau^{q_B}\bar{h}_\lambda(\tau),
	\end{align}
where
\begin{equation*}
	\bar{h}_\lambda(\tau):=\alpha_k\tau^{p^+_B -q_B}+\beta_k-\gamma_k\lambda\tau^{s^-_B -q_B}
\end{equation*}
with $\alpha_k:=C_p\delta_k^{p^+_B}$, $\beta_k:=C_B\delta_k^{q_B}$ and $\gamma_k:=(s^+)^{-1}\delta_k^{-s^-_B}$. Denote 
\begin{equation}\label{T^*}
	T^*:=\left[\frac{(s_B^--q_B)\gamma_k\lambda}{(p_B^+-q_B)\alpha_k}\right]^{\frac{1}{p_B^+-s_B^-}}.
\end{equation}
Then, for $$\lambda>C_1(B)\delta_k^{2s_B^-}$$
 with $C_1(B):=\left(\frac{C_B(p_B^+-q_B)}{p_B^+-s_B^-}\right)^{\frac{p_B^+-s_B^-}{p_B^+-q_B}}\left(\frac{C_p(p_B^+-q_B)}{s_B^--q_B}\right)^{\frac{s_B^--q_B}{p_B^+-q_B}}s^+$, it holds that
\begin{equation}\label{bar-h(T*)}
	\bar{h}_\lambda(T^*)=\beta_k-\frac{p_B^+-s_B^-}{p_B^+-q_B}\left(\frac{s_B^--q_B}{p_B^+-q_B}\right)^{\frac{s_B^--q_B}{p_B^+-s_B^-}}\alpha_k^{-\frac{s_B^--q_B}{p_B^+-s_B^-}}\gamma_k^{\frac{p_B^+-q_B}{p_B^+-s_B^-}}\lambda^{\frac{p_B^+-q_B}{p_B^+-s_B^-}}<0.
\end{equation}
From \eqref{T^*} we have
\begin{equation}\label{tilde-lambda_1}
	T^*>\delta_k \Longleftrightarrow \lambda>C_2(B)\delta_k^{2p_B^+}
\end{equation}
with $C_2(B):=\frac{(p_B^+-q_B)C_ps^+}{s_B^--q_B}$. Set
\begin{equation}\label{S4.lambda_*}
	\lambda_k:=C_3(B)\delta_k^{2p_B^+},
\end{equation}
where $C_3(B):=\max\left\{1,a_0,C_1(B),C_2(B)\right\}$ with $a_0$ given by \eqref{PT.Sand.lbd1}. Then, we have that $\lambda_k>\max\left\{1,a_0\right\}$ since $\delta_k>1$. Moreover, for $\lambda>\lambda_k$ it holds that $T^*>\delta_k$ due to \eqref{tilde-lambda_1} and therefore,
\begin{align*}\label{T<0}
	T_\lambda(u)\leq (T^*)^{q_B}\bar{h}_\lambda(T^*)=:-\epsilon_k<0, \ \ \forall\, u\in \partial B_{T^*}\cap E_k
\end{align*}
 in view of \eqref{T} and \eqref{bar-h(T*)}. Thus, $\partial B_{T^*}\cap E_k\subset \{u \in E: T_\lambda(u) \leq -\epsilon_k\}=:T_\lambda^{-\epsilon_k}$.  Note that $T_\lambda^{-\epsilon_k}\in\Sigma$ due to $(\textup{A0})$ and the definition of $T_\lambda$. Hence, we arrive at
 $$\gamma( T_\lambda^{-\epsilon_k})\geq \gamma(\partial B_{T^*}\cap E_k)=k
 $$
 (see \cite[Proposition 7.7]{Rab1986}). In other words, $T_\lambda^{-\epsilon_k}\in \Sigma_k$ and thus, 
 $$
 c_k(\lambda,\theta)= \inf_{A\in
 	\Sigma_k} \sup_{u \in A}\, T_{\lambda}(u) \leq
 \sup_{u\in T_{\lambda}^{-\epsilon_k}}\, T_{\lambda}(u) \leq
 -\epsilon_k <0.
 $$
Finally, it is obvious that this sequence $\{\lambda_k\}_{k=1}^\infty$ satisfies $\lambda_k<\lambda_{k+1}$ for all $k\in\N$. The proof is complete.
 
	\end{proof}
Let $\{\lambda_k\}_{k=1}^\infty$ be defined as in \eqref{S4.lambda_*}. Note that with $\lambda_0(\theta)$  defined by \eqref{S4.lambda0(theta)} we have
\begin{equation*}
	\lambda_k=\lambda_0(\theta)=c_0\theta^{\frac{s^+-p^-}{t^+-p^-}} \ \Longleftrightarrow \  \theta=[c_0^{-1}C_3(B)]^{\frac{t^+-p^-}{s^+-p^-}}\delta_k^{-\frac{2p_B^+(t^+-p^-)}{p^--s^+}}.
\end{equation*}
Moreover, for $\theta\in (0,1)$ and $\lambda>1$, the inequality 
\begin{align}\label{S4.lambda}
	0<\left(\frac{1}{p^+}-\frac{1}{t^-}\right)S_b^N&\min\left\{\theta^{-\frac{1}{h^+-1}},\theta^{-\frac{1}{h^--1}}\right\}\notag\\
	&-K\max\left\{\theta^{-\frac{1}{\ell^+-1}},\theta^{-\frac{1}{\ell^--1}}\right\}\max\left\{|\lambda|^{\frac{\ell^+}{\ell^+-1}},|\lambda|^{\frac{\ell^-}{\ell^--1}}\right\}
\end{align}
is equivalent to
\begin{equation}\label{S4.bar-br-lambda0(theta)}
	\lambda<\left[K^{-1}\left(\frac{1}{p^+}-\frac{1}{t^-}\right)S_b^N\right]^{\frac{\ell^--1}{\ell^-}}\theta^{-\frac{\ell^--h^+}{(h^+-1)\ell^-}}=e_0\theta^{-\frac{\ell^--h^+}{(h^+-1)\ell^-}}=:\bar{\lambda}_0(\theta)
\end{equation}
with $e_0:=\left[K^{-1}\left(\frac{1}{p^+}-\frac{1}{t^-}\right)S_b^N\right]^{\frac{\ell^--1}{\ell^-}}$. We also note that
\begin{equation*}
	\lambda_k=\bar{\lambda}_0(\theta) \ \Longleftrightarrow \  \theta=[e_0^{-1}C_3(B)]^{\frac{(h^+-1)\ell^-}{h^+-\ell^-}}\delta_k^{-\frac{2p_B^+(h^+-1)\ell^-}{\ell^--h^+}}.
\end{equation*}
Thus, by taking
\begin{equation}\label{S4.theta*}
	\theta_k:=\min\left\{1,\theta_*,[c_0^{-1}C_3(B)]^{\frac{t^+-p^-}{s^+-p^-}},[e_0^{-1}C_3(B)]^{\frac{(h^+-1)\ell^-}{h^+-\ell^-}}\right\}\delta_k^{-\kappa}
\end{equation}
with $\kappa:=\max\left\{\frac{2p_B^+(t^+-p^-)}{p^--s^+},\frac{2p_B^+(h^+-1)\ell^-}{\ell^--h^+}\right\}$ and setting
\begin{equation}\label{S4.lambda^*}
	\lambda^*(\theta):=\min \left\{\lambda_0(\theta),\bar{\lambda}_0(\theta)\right\},
\end{equation}
we deduce from the definitions of $\lambda_0(\theta)$ and $\bar{\lambda}_0(\theta)$ in \eqref{S4.lambda0(theta)} and \eqref{S4.bar-br-lambda0(theta)} that $\theta_k>0$ is independent of $\lambda$ and for any $\theta\in (0,\theta_k)$, 
\begin{equation*}
	\theta<\min\{1,\theta_*\} \ \text{and}\ \lambda_k<\lambda^*(\theta).
\end{equation*}
We are now ready to prove Theorem~\ref{Theo.Sandwich-infty}.
\begin{proof}[\textbf{Proof of Theorem~\ref{Theo.Sandwich-infty}}]
	Let $k\in\N$ and let $\lambda_k$ and $\theta_k$ be defined as in \eqref{S4.lambda_*} and \eqref{S4.theta*}, respectively. Let $\theta\in (0,\theta_k)$ be given and let $\lambda\in (\lambda_k,\lambda^*(\theta))$ with $\lambda^*(\theta)$ given by \eqref{S4.lambda^*}. Let $\{c_n(\lambda,\theta)\}_{n=1}^\infty$ be defined as in \eqref{S4.c_k}; for simplicity of notation, we write $c_n$ in place of $c_n(\lambda,\theta)$. By the definitions of $\{\theta_n\}_{n=1}^\infty$ and  $\{\lambda_n\}_{n=1}^\infty$, for any $n\in \{1,\cdots,k\}$ we have that $0<\theta<\theta_n<1$ and $\lambda_n<\lambda<\lambda^*(\theta))$ as well. By Lemma~\ref{Le.c_k<0}, it follows that $-\infty<c_n<0.$ From this and \eqref{S4.lambda}-\eqref{S4.bar-br-lambda0(theta)} we infer that $T_\lambda$ satisfies  the $\textup{(PS)}_{c_n}$ condition in view of Lemmas~\ref{le.sandwich.PS} and \ref{T_lambda(u)<0}. Suppose that $c_n=\cdots=c_k=c$ for some $n\in \{1,\cdots,k\}$. Then, it follows from the fact that $T_\lambda$ satisfies the $\textup{(PS)}_{c_n}$ condition  that $K_{c}:=\{u \in E\backslash\{0\}: T_\lambda'(u)=0  \text{ and } T_\lambda(u)=c\}$ is a compact set. By a standard argument using the deformation lemma and properties of genus (see e.g., \cite[Proof of Theorem 1]{BBF.2021}) we obtain 
	$$ \gamma(K_{c})\geq k-n+1.$$
	This implies that either $\{c_n\}_{n=1}^k$ are $k$ distinct critical values of $T_\lambda$ or $K_{c_k}$ admits infinitely many points. From this and the fact that  $c_1\leq c_2\leq \cdots\leq c_k<0$, $\Psi$ admits at least $k$ pairs of critical points $\{\pm u_n\}_{n=1}^k$ with $\Psi(\pm u_n)<0$ for $n=1,\cdots,k$ in view of Lemma~\ref{T_lambda(u)<0}. Consequently, problem~\eqref{Eq.Sandwich} admits at least $k$ pair of solutions $\{\pm u_n\}_{n=1}^k$ with negative energy. The proof is complete.

\end{proof}
\begin{remark}\label{compareBBF}\rm
	As we saw in the proof of Theorem~\ref{Theo.Sandwich-infty}, the range of $\theta$ to get at least $k$ pair of solutions is $(0,\theta_k)$. Since $\theta_k\to 0$ as $k\to\infty$, we cannot conclude that problem~\eqref{Eq.Sandwich} admits a sequence of solution with a certain range of $\theta$. This similarly happens with the proof of \cite[Theorem 1]{BBF.2021}. Indeed, \cite[Eq. (70)]{BBF.2021} gives the range of $\|K\|_{\infty}$ of the form $\|K\|_\infty<Cd_j^\alpha$ for some positive constants $C,\alpha$ independent of $d_j$, where $d_j$ is given by \cite[Eq. (67)]{BBF.2021} and possibly satisfies that $d_j\to 0$ as $j\to\infty$.
\end{remark}

\subsection{Nontrivial nonnegative solutions}
In this subsection, we will prove Theorem~\ref{Theo.Sandwich}. We will apply the Ekeland variational principle to determine critical points of the functional $J$ given by \eqref{J}, that are in turn nonnegative solutions  to problem~\eqref{Eq.Sandwich}. Throughout this subsection, in addition to $\textup{(S)}$, $\textup{(W)}$ and $\textup{(L1)}$, we always assume that $\textup{(L2)}$ holds.

We start with the following geometry of $J$.
\begin{lemma}\label{le.sandwhich.local_geo}
	For each given $\lambda>\lambda_\star$, there exists $\widetilde{\theta}_\ast>0$ such that for any $\theta\in (-\widetilde{\theta}_\ast,\widetilde{\theta}_\ast)$, there exist $r,\rho>0$ such that
	$$\underset{u\in \partial B_r}{\inf} J(u)\geq \rho>0> \underset{u\in B_r}{\inf} J(u).$$
\end{lemma}
\begin{proof} 	
	Fix $\lambda>\lambda_\star$ with $\lambda_\star$ given by \eqref{lamda*} and let $\theta\in\R$ . Since $\lambda>\lambda_\star$ we can find $\phi\in \mathcal{A}$ such that
	\begin{equation}\label{PL.Sand.lambda}
		\max_{\eta\in [s_B^-,s_B^+]}\, \frac{C^*(\eta)}{\int_B \frac{m(x)}{s(x)}\phi_+^{s(x)}\diff x}\left(\int_B |\nabla \phi|^{p_B^+}\diff x\right)^{\frac{\eta-q_B}{p_B^+-q_B}}\left(\int_B |\nabla \phi|^{q_B}\diff x\right)^{\frac{p_B^+-\eta}{p_B^+-q_B}}<\lambda.
	\end{equation}
	Fix such $\phi$. By \eqref{A6} and $\textup{(L2)}$, we have
	\begin{align*}
		J(\tau\phi)\leq &C_p\left(\int_B |\nabla \phi|^{p_B^+}\diff x\right)\tau^{p_B^+}+C_B\left(\int_B |\nabla \phi|^{q_B}\diff x\right)\tau^{q_B}\\
		&-\lambda\left(\int_B \frac{m(x)}{s(x)}\phi_+^{s(x)}\diff x\right)\tau^{s_B(\tau)}-\theta\left(\int_B \frac{b(x)}{t(x)}(\tau\phi_+)^{t(x)}\diff x\right),\quad \forall \tau>0,
	\end{align*}
	where $s_B(\tau)$ is determined by $\phi$ and $\tau$ as in $\textup{(L2)}$. Thus,
	\begin{equation}\label{PL.Sand.J}
		J(\tau\phi)\leq y_\lambda(\tau)-\left(\int_B \frac{b(x)}{t(x)}(\tau\phi_+)^{t(x)}\diff x\right)\theta,\quad \forall \tau\geq 0,
	\end{equation}
	where $$y_\lambda(\tau):=\alpha_1\tau^{p_B^+}-\beta_1\lambda \tau^{s_B(\tau)}+\gamma_1\tau^{q_B}$$
	with $\alpha_1:=C_p\int_B |\nabla \phi|^{p_B^+}\diff x>0,\, \beta_1:=\int_B \frac{m(x)}{s(x)}\phi_+^{s(x)}\diff x>0$ and $\gamma_1:=C_B\int_B |\nabla \phi|^{q_B}\diff x>0.$ 
	We claim that there exists $\tau_0=\tau_0(\lambda)\in (0,\infty)$ such that
	\begin{equation}\label{PL.Sand.xi_0}
		y_\lambda(\tau_0)<0.
	\end{equation}
	Indeed, we write
	\begin{equation}\label{PL.Sand.g*}
		y_\lambda(\tau)=\tau^{q_B}\, \overline{y}_\lambda(\tau)\ \text{for}\ \tau>0,
	\end{equation}
	where $\overline{y}_\lambda(\tau):=\alpha_1\tau^{p_B^+-q_B}-\beta_1\lambda \tau^{s_B(\tau)-q_B}+\gamma_1$ for $\tau>0.$ It is clear that $s_B(\cdot): (0,\infty)\to [s_B^-,s_B^+]$ is a continuous function and hence, there exist $\tau_0>0$ such that
	\begin{equation}\label{PL.Sand.tau0}
		\tau_0=\left(\frac{s_B(\tau_0)-q_B}{p_B^+-q_B}\alpha_1^{-1}\beta_1\lambda\right)^{\frac{1}{p_B^+-s_B(\tau_0)}}.
	\end{equation}
	We have 
	\begin{align*}
		\overline{y}_\lambda\left(\tau_0\right)=\gamma_1-\frac{p_B^+-s_B(\tau_0)}{p_B^+-q_B}\left(\frac{s_B(\tau_0)-q_B}{p_B^+-q_B}\right)^{\frac{s_B(\tau_0)-q_B}{p_B^+-s_B(\tau_0)}}\alpha_1^{\frac{q_B-s_B(\tau_0)}{p_B^+-s_B(\tau_0)}}\beta_1^{\frac{p_B^+-q_B}{p_B^+-s_B(\tau_0)}}\lambda^{\frac{p_B^+-q_B}{p_B^+-s_B(\tau_0)}}.
	\end{align*}
	Thus, $\overline{y}_\lambda\left(\tau_0\right)<0$ due to \eqref{PL.Sand.lambda}; hence,
	\begin{equation}\label{PL.Sand.g(tau0)}
		y_\lambda\left(\tau_0\right)<0,
	\end{equation}
	due to \eqref{PL.Sand.g*}. 
	Then, from \eqref{PL.Sand.g(tau0)} and \eqref{PL.Sand.J} we have
	\begin{equation}\label{PL.Sandwich.loc.behaviour.xi_0}
		J(\tau_0\phi)\leq y_\lambda(\tau_0)-\left(\int_B \frac{b(x)}{t(x)}(\tau_0\phi_+)^{t(x)}\diff x\right)\theta<0
	\end{equation}
	provided 
	\begin{equation}\label{PL.Sand.theta1}
		\theta>-\bar{\theta}_\ast,
	\end{equation}	
	where $\bar{\theta}_\ast:=-\frac{y_\lambda(\tau_0)}{\int_B \frac{b(x)}{t(x)}(\tau_0\phi_+)^{t(x)}\diff x}>0$.
	Set
	\begin{equation}\label{PL.Sand.r}
		r_\lambda:=\max\left\{1+\tau_0\|\phi\|,\left(2p^+C_1^{s^+}(s^-)^{-1}\lambda\right)^{\frac{1}{p^--s^+}}\right\}
	\end{equation}
	and
	\begin{equation}\label{PL.Sand.theta2}
		\widetilde{\theta}_\ast:=\min\left\{\bar{\theta}_\ast,\frac{t^- r_\lambda^{p^--t^+}}{2p^+C_1^{t^+}}\right\},
	\end{equation}
	where $C_1$ is given in \eqref{S4.Est.ints}. Then, for any $\theta\in (-\widetilde{\theta}_\ast,\widetilde{\theta}_\ast)$, by choosing $r=r_\lambda$ and $\rho=\frac{C_1^{t^+} r_\lambda^{t^+}}{t^-}\left(\widetilde{\theta}_\ast-|\theta|\right)$ and utilizing \eqref{S4.Est.ints} and \eqref{PL.Sand.r}-\eqref{PL.Sand.theta2} we have
	\begin{align*}
		J(u)&\geq \frac{1}{p^+}\|u\|^{p^-}-\frac{C_1^{t^+}\lambda}{s^-}\|u\|^{s^+}-\frac{C_1^{t^+}|\theta|}{t^-}\|u\|^{t^+}\\
		&\geq \frac{1}{2p^+}\|u\|^{p^-}-\frac{C_1^{t^+}|\theta|}{t^-}\|u\|^{t^+}=\frac{C_1^{t^+}r_\lambda^{t^+}}{t^-}\left(\widetilde{\theta}_\ast-|\theta|\right),\ \forall u\in\partial B_r,
	\end{align*}
	i.e., 
	\begin{equation}\label{PL.Sandwich.Geo.partial B_r}
		J(u)\geq\rho,\ \ \forall u\in\partial B_r.
	\end{equation}
	Finally, note that $\tau_0\phi\in B_r$ by the choice of $r$ and \eqref{PL.Sand.r};  hence, \eqref{PL.Sandwich.loc.behaviour.xi_0} yields 
	$$\underset{u\in B_r}{\inf} J(u)\leq J(\tau_0\phi)<0.$$
	This and \eqref{PL.Sandwich.Geo.partial B_r} complete the proof.
	
\end{proof}
We are now in a position to give a proof of Theorem~\ref{Theo.Sandwich}. 


\begin{proof}[\textbf{Proof of Theorem~\ref{Theo.Sandwich}}] Let $\lambda>\lambda_\star$ be arbitrary and fixed. Note that for $\theta\in (0,1)$, the right-hand side of \eqref{PS_c2} can be rewritten as
	\begin{align*}
		C_2\theta^{-\frac{1}{h^+-1}}-C_3(\lambda)\theta^{-\frac{1}{\ell^--1}}=\theta^{-\frac{1}{h^+-1}}\left[C_2-C_3(\lambda)\theta^{\frac{1}{h^+-1}-\frac{1}{\ell^--1}}\right],
	\end{align*}	
	where $C_2, C_3(\lambda)>0$ are positive constants independent of $\theta.$	Note that $\frac{1}{h^+-1}-\frac{1}{\ell^--1}>0$ due to the assumption $\left(\frac{t}{p}\right)^+<\left(\frac{t}{s}\right)^-$. Hence, we can choose $\theta_\star\in \big(0,\min\{1,\widetilde{\theta}_\ast\}\big)$, where $\widetilde{\theta}_\ast$ is as in Lemma~\ref{le.sandwhich.local_geo}, such that
	\begin{align}\label{PT.sandwich.theta}
		0<&\left(\frac{1}{p^+}-\frac{1}{t^-}\right)S_b^N\min\left\{\theta^{-\frac{1}{h^+-1}},\theta^{-\frac{1}{h^--1}}\right\}\notag\\
		&-K\max\{\theta^{-\frac{1}{\ell^+-1}},\theta^{-\frac{1}{\ell^--1}}\}\max\left\{|\lambda|^{\frac{\ell^+}{\ell^+-1}},|\lambda|^{\frac{\ell^-}{\ell^--1}}\right\}, \ \forall\, \theta\in (0,\theta_\star).
	\end{align}
	Let $\theta\in [0,\theta_\star)$. By Lemma~\ref{le.sandwhich.local_geo}, there exist $r,\rho>0$ such that
	\begin{equation}\label{PT.sandwich.r-rho}
		\underset{u\in \partial B_r}{\inf} J(u)\geq \rho>0> \underset{u\in B_r}{\inf} J(u)=:c.
	\end{equation}	
	Then, arguing as in \cite[Proof of Theorem 3.1]{HS.TJM2015} in which the Ekeland variational principle was applied, we find a $\textup{(PS)}_c$-sequence $\{u_n\}_{n=1}^\infty.$ On the other hand, from \eqref{PT.sandwich.theta} and \eqref{PT.sandwich.r-rho}, $J$ satisfies the $\textup{(PS)}_c$ condition in view of Lemma~\ref{le.sandwich.PS}, and hence $u_n\to u$ in $E$ as $n\to\infty.$ Thus,
	$$J'(u)=0\ \ \text{and}\ \ J(u)=c<0.$$
	That is, $u$ is a nontrivial nonnegative solution to problem~\eqref{Eq.Sandwich}.
	
	
\end{proof}
Let us conclude this section by pointing out that Theorem~\ref{Theo.Sandwich} remains valid for subcritical problem. More precisely, consider the following problem:
\begin{align}\label{Eq.Sandwich.sub}
	\begin{cases}
		-\operatorname{div}a(x,\nabla u)=\lambda m(x)|u|^{s(x)-2}u+\mu\, \omega(x)|u|^{r(x)-2}u \quad \text{in } \Omega,\\
		u=0\quad \text{on } \partial \Omega, 
	\end{cases}
\end{align}
where $(\textup{A0})-(\textup{A5})$ hold; $p,q,s,r\in C_+(\overline{\Omega})$ such that $s^+<p^-\leq p^+<r^-$ and $r(x)<p^*(x)$ for all $x\in\overline{\Omega}$; and the weights $m$ and $\omega$ satisfy the following assumption.
\begin{itemize}
	\item [$\textup{(A)}$] $m\in L^{\frac{p^\ast(\cdot)}{p^\ast(\cdot)-s(\cdot)}}(\Omega)$, $\omega\in L_+^{\frac{p^\ast(\cdot)}{p^\ast(\cdot)-r(\cdot)}}(\Omega)$ and there exists a ball $B\subset\Omega$ satisfying
	\begin{itemize}
		\item [(i)] $q_B<s_B^-$;
		\item [(ii)] $|\{x\in B:\, m(x)>0 \}|>0$;
		\item [(iii)] there exists $\phi\in C_c^\infty(B)$ such that  $\int_B \frac{m(x)}{s(x)}\phi_+^{s(x)}\diff x>0$ and for each $\tau>0$, there exists $s_B(\tau)\in [s_B^-,s_B^+]$ such that
		$$\int_B \frac{m(x)}{s(x)}\phi_+^{s(x)}\tau^{s(x)}\diff x=\tau^{s_B(\tau)}\int_B \frac{m(x)}{s(x)}\phi_+^{s(x)}\diff x.$$ 
	\end{itemize} 
\end{itemize}
\begin{theorem}\label{Theo.Sandwich.Subcritical}
	Let $(\textup{A0})-(\textup{A5})$ and $\textup{(A)}$ hold with $s^+<p^-\leq p^+<r^-$ and $r(x)<p^*(x)$ for all $x\in\overline{\Omega}$. Then, for each given $\lambda>\lambda_\star$ with $\lambda_\star$ given by \eqref{lamda*}, there exists $\mu_\ast>0$ such that for any $\mu\in (-\mu_\ast,\mu_\ast)$, problem~\eqref{Eq.Sandwich.sub} has a nontrivial nonnegative solution and has two nontrivial nonnegative solutions for any $\mu\in (0,\mu_\ast)$. Moreover, if $r(\cdot)$ is constant in $\Omega$, it suffices to assume $\omega\in L^{\frac{p^\ast(\cdot)}{p^\ast(\cdot)-r(\cdot)}}(\Omega)$ and $|\{x\in \Omega:\, \omega(x)>0 \}|>0$ instead of $\omega\in L_+^{\frac{p^\ast(\cdot)}{p^\ast(\cdot)-r(\cdot)}}(\Omega)$. 
\end{theorem}
This result can be proved using the same method as in \cite[Theorem 1.2]{HS.AML21} so we omit it.

\section{Comment on the regularity of solutions}\label{Comments}
Let functions  $d$ in $(\textup{F1})$, $a_j$ in $(\textup{F4})$
or $m$ in $(\textup{W})$ or $m, \omega$ in $(\textup{A})$ be in the class $L^\infty(\Omega)$. Then by \cite{HKWZ}, all aforementioned solutions are of class $L^\infty(\Omega)$. Moreover, we have that:
\begin{itemize}
	\item [(i)] when $p\in C^{1}(\overline{\Omega})$ and $A(x,u,\nabla u):=a(x,\nabla u)$ additionally satisfies Assumption $(A^k)$ in \cite{Fan2007}, for instance, the $p(\cdot)$-Laplacian or the generalized mean curvature operator, any aforementioned solution is also of class $C^{1,\beta}(\overline{\Omega })$ and the convergence in Theorem~\ref{Theo.Sublinear} is in the $C^1(\overline{\Omega })$-topology (see \cite[Theorem 1.2]{Fan2007}); 
	\item [(ii)]when $p(\cdot)$ is constant and $A(x,u,\nabla u):=a(x,\nabla u)$ additionally satisfies conditions (1.10a)-(1.10c) in \cite{Liberman1991}, for instance, the case of the multiple $p$-Laplacian, any aforementioned solution is also of class $C^{1,\beta}(\overline{\Omega })$ and the convergence in Theorem~\ref{Theo.Sublinear} is in the $C^1(\overline{\Omega })$-topology (see \cite[Theorem 1.7 and the comments below Theorem 1.7]{Liberman1991}).
\end{itemize}

\subsection*{Acknowledgements}
The first author was supported by University of Economics Ho Chi Minh City, Vietnam. The second author was supported by the National
Research Foundation of Korea Grant funded by the Korea Government (MEST)
(NRF-2021R1I1A3A0403627011).


\begin{thebibliography}{99}
	
\bibitem{BBF.2021}L. Baldelli, Y. Brizi, R. Filippucci, Multiplicity results for $(p, q)$-Laplacian equations with critical exponent in $\R^N$ and negative energy, Calc. Var. (2021) 60:8.

\bibitem{Bonder} J.F. Bonder, A. Silva,
Concentration-compactness principle for variable exponent spaces and applications, Electron. J. Differential Equations 141 (2010) 1--18.

\bibitem{Brezis-book} H. Brezis,
Functional analysis, Sobolev spaces and partial differential equations, Universitext, Springer, New York, 2011.

\bibitem{de Figueiredo.2003}D.G. de Figueiredo, J.-P. Gossez, P. Ubilla, Local superlinearity and sublinearity for indefinite semilinear elliptic problems, J. Funct. Anal. 199 (2) (2003) 452--467.

\bibitem{de Figueiredo.2006} D.G. de Figueiredo, J.-P. Gossez, P. Ubilla, Multiplicity results for a family of semilinear elliptic problems under local superlinearity and sublinearity, J. Eur. Math. Soc. 8 (2) (2006) 269--286.

\bibitem{de Figueiredo.2009} D.G. de Figueiredo,  J.-P. Gossez, P. Ubilla,
Local ``superlinearity'' and ``sublinearity'' for the $p$-Laplacian, J. Funct. Anal. 257 (2009) 721--752.

\bibitem{Diening} L. Diening,  P. Harjulehto, P. H\"ast\"o, M. R\r{u}\v{z}i\v{c}ka,
Lebesgue and Sobolev spaces with  variable exponents,
Lecture Notes in Mathematics  2017, Springer-Verlag, Heidelberg, 2011.
		
\bibitem{Fan2007} X. Fan,
Global $C^{1,\alpha}$ regularity for variable exponent elliptic equations in divergence form, J. Differential Equations 235 (2007) 397--417.

\bibitem{FZ} X. Fan, D. Zhao, On the spaces $L^{p(x)}(\Omega)$ and $W^{m,p(x)}(\Omega)$, J. Math. Anal. Appl. 263 (2001) 424--446.

\bibitem{Fan-Zhang-Zhao.2005} X. Fan, Q. Zhang. D. Zhao, Eigenvalues of $p(x)$-Laplacian Dirichlet problem, J. Math. Anal. Appl. 302 (2005) 306--317.
	
	\bibitem{Figueredo2013} G.M. Figueiredo, Existence and multiplicity of solutions for a class of $p\&q$ elliptic problems with critical exponent, Math. Nachr. 286 (11--12) (2013) 1129--1141. 
	
\bibitem{HS.TJM2015} K. Ho, I. Sim,
	Existence and multiplicity of solutions for degenerate $p(x)$-Laplace equations involving concave-convex type nonlinearities with two parameters, Taiwanese J. Math. 19 (5) (2015) 1469--1493.
	
	\bibitem{HHS}H.H. Ha, K. Ho, I. Sim, Infinitely many solutions for a generalized $p(\cdot)$‐Laplace equation involving Leray-Lions type operators, Math. Methods Appl. Sci., In press, DOI: 10.1002/mma.7246.
	
\bibitem{HKWZ}K. Ho, Y.-H. Kim, P. Winkert, C. Zhang, The boundedness and H\" older continuity of weak solutions to elliptic equations involving variable exponents and critical growth, J. Differential Equations 313 (2022) 503--532
	
	\bibitem{HNT} K. Ho, L.C. Nhan, L.X. Truong, A-priori bound and H\"older continuity of solutions to degenerate elliptic equations with variable exponent, Manuscript submitted for publication.
	
	\bibitem{HS.NA2016}K. Ho, I. Sim,
	On degenerate $p(x)$-Laplace equations involving critical growth with two parameters, Nonlinear Anal. 132 (2016), 95--114.
	
\bibitem{HS.AML21}K. Ho, I. Sim, An existence result for $(p, q)$-Laplace equations involving sandwich-type and critical growth, Appl. Math. Lett. 111 (2021), 106646.
	
	\bibitem{Kawohl} B. Kawohl, M. Lucia, S. Prashanth,
	Simplicity of the principal eigenvalue for indefinite quasilinear
		problems, Adv. Differ. Equ. 12 (4) (2007) 407--434.
	
\bibitem{Kaj.2006} R. Kajikiya, A critical point theorem related to the symmetric mountain pass lemma and its applications to elliptic equations, J. Funct. Anal. 225 (2005) 352--370.

\bibitem{Kim-Sim} Y.-H. Kim, I. Sim,
Existence of solutions and positivity of the infimum eigenvalue for degenerate elliptic equations with variable exponents,  Discrete Contin. Dyn. Syst. Supplement (2013), 695--707.
	
\bibitem{Kom-Kaj}Y. Komiya, R. Kajikiya, 
Existence of infinitely many solutions for the $(p,q)$-Laplace equation, NoDEA Nonlinear Differential Equations Appl. 23 (2016) 49, 23pp.

\bibitem{KR} O. Kov\'{a}\v{c}ik, J. R\'{a}kosn\'{i}k,
On spaces $L^{p(x)}$ and $W^{k,p(x)}$, Czechoslovak Math. J. 41
(1991) 592--618.

\bibitem{Lang2016} J. Lang, O. M\'{e}ndez, Stability of the Norm-Eigenfunctions of the $p(\cdot)$-Laplacian, Integr. Equ. Oper. Theory 85 (2016), 245--257.

\bibitem{LKV.2009} V.K. Le, On a sub-supersolution method for variational inequalities with Leray-Lions operators in variable exponent spaces, Nonlinear Anal. 71 (2009), 3305--3321.

\bibitem{Liberman1991}G.M. Lieberman, The natural generalization of the natural conditions of Ladyzhenskaya and Ural’tseva for elliptic equations, Commun. Partial Differ. Equ. 16 (1991) 311--361.

\bibitem{Mih-Rad}M. Mihailescu, V. R\v adulescu, Continuous spectrum for a class of nonhomogeneous differential operators, Manuscripta Math. 125 (2008) 157--167.

\bibitem{Rab1986} P.H. Rabinowitz, 
Minimax methods in critical point theory with applications to differential equations, CBMS regional conference series in mathematics, vol. 65, American Mathematical Society, Providence, RI, 1986.

\bibitem{Radul} V. R\v adulescu,
Nonlinear elliptic equations with variable exponent: old and new,  Nonlinear Anal. 121 (2015) 336--369.

\bibitem{Sil-Xav.2003}E.A.B. Silva, M.S. Xavier, Multiplicity of solutions for quasilinear elliptic problems involving critical Sobolev exponents, Ann. Inst. H. Poincar\'e Anal. Non Lin\'eaire 20 (2003), 341--358.


\end{thebibliography}
\end{document}